\documentclass[12pt]{amsart}

\usepackage{amsmath,amssymb,enumerate}

\usepackage{epsfig,fancyhdr,color}%,showkeys,amsmidx

\usepackage{amssymb}
\usepackage{amsmath,amsthm}
\usepackage{latexsym}
\usepackage{amscd}
\usepackage{psfrag}
\usepackage{graphicx}
\usepackage[latin1]{inputenc}
\usepackage[all]{xy}

\definecolor{NoteColor}{rgb}{1,0,0}

\newtheorem{theorem}{\rm\bf Theorem}[section]
\newtheorem{proposition}[theorem]{\rm\bf Proposition}
\newtheorem{lemma}[theorem]{\rm\bf Lemma}
\newtheorem{corollary}[theorem]{\rm\bf Corollary}
\newtheorem*{theorem 1}{\rm\bf Proposition 1}
\newtheorem*{theorem 2}{\rm\bf Proposition 2}

\theoremstyle{definition}

\theoremstyle{remark}

\def\interieur#1{\mathord{\mathop{\kern 0pt #1}\limits^\circ}}

\title[Non-orientable surfaces]{Hyperbolic metrics, measured foliations and pants decompositions for non-orientable surfaces}

\author{A.\ Papadopoulos}
\address{Institut de Recherche Math\'ematique Avanc\'ee\\Universit\'e de Strasbourg\\7 rue Ren\'e Descartes\\F- 67084 Strasbourg, France and The Erwin Schr\"odinger International Institute for Mathematical Physics, Boltzmanngasse 9,  A-1090 Wien, Austria}
\email{athanase.papadopoulos{\char'100}math.unistra.fr}

\author{R.\ C.\ Penner}
\address{Center for the Topology and Quantization of Moduli Spaces\\Department of Mathematics\\
Aarhus University\\
DK-8000 Aarhus C, DK\\
 Math and Physics Departments, Caltech\\
Pasadena, CA 91125 USA and The Erwin Schr\"odinger International Institute for Mathematical Physics, Boltzmanngasse 9,  A-1090 Wien, Austria}
\email{rpenner{\char'100}caltech.edu}

\thanks{It is a pleasure for both authors to thank the Erwin Schr\"odinger Insitute for hosting and the GEometric structures And Representation varieties Network of the US National Science Foundation for partial support of an excellent trimester in Vienna when this work began. The first-named author was partially supported by the ANR French project ModGroup and the second- by QGM (Centre for the Quantum Geometry of Moduli Spaces) funded by the
Danish National Research Foundation.}

\begin{document}

\begin{abstract} We provide analogues for non-orientable surfaces with or without boundary or punctures of several basic theorems in the setting of the Thurston theory of surfaces which were developed so far only in the case of orientable surfaces. Namely, we provide natural analogues for non-orientable surfaces of the Fenchel-Nielsen theorem on the parametrization of the Teichm\"uller space of the surface, the Dehn-Thurston theorem on the parametrization of measured foliations in the surface, and the Hatcher-Thurston theorem, which gives a complete minimal set of moves between pair of pants decompositions of the surface.  For the former two theorems, one in effect drops the twisting number for any curve
in a pants decomposition which is 1-sided, and for the latter, two further elementary moves on pants decompositions are
added to the two classical moves.
\end{abstract}

\maketitle

%\bigskip

\smallskip

\smallskip

\maketitle
\tableofcontents

\section{Introduction}

The study of surfaces together with their associated topological, combinatorial and geometric constructs and the action of surface mapping class groups on the corresponding spaces has been carried out essentially in the case of orientable surfaces.  Our aim in this paper is to present three important results from this body of work, namely, the Fenchel-Nielsen theorem on the parametrization of the Teichm\"uller space\footnote{
The Riemann surface structure in the non-orientable case is given by charts with transition functions which are either analytic or anti-analytic and can be regarded as an appropriate class of either complete finite-area hyperbolic or conformal metric on the underlying smooth manifold.  These are sometimes called ``Klein surfaces'' \cite{Klein1,Klein2,Seppala}.
} of the surface, the Dehn-Thurston theorem on the parametrization of the space of measured foliations in the surface
and the Hatcher-Thurston theorem, which provides a complete minimal set of moves between pair of pants decompositions of the surface, in the setting of possibly non-orientable surfaces.

We shall find, in effect, that if ${\mathcal P}$ is a pants decomposition of a possibly non-orientable surface $S$, i.e., each component of $S-\cup{\mathcal P}$ is a pair of pants, then the hyperbolic lengths (transverse measures respectively)
of all the curves in ${\mathcal P}$ together with twisting numbers defined only for the 2-sided curves in ${\mathcal P}$ give Fenchel-Nielsen coordinates (Dehn-Thurston coordinates respectively) on the Teichm\"uller space of $S$ (space of measured foliations in $S$ respectively).  
The twisting numbers we introduce here for both orientable and non-orientable surfaces provide a new point of view on the usual definitions,
see, e.g., \cite{FN,FN2,Penner-thesis,Thurston} in the orientable case.  In fact, our new incarnations of twisting numbers depend upon orientations 
of the curves in a pants decomposition in order to conveniently specify the right or left sense of the twisting as well as a basepoint in each pants curve.
We shall be more specific on this below.

We shall also find that the four moves on pants decompositions indicated in Figure \ref{fig:elementary} give a complete minimal set of moves acting transitively on the set of all pants decompositions of $S$.  Here, as well as in other figures of
this paper, the discs with asterisks drawn within them represent cross caps; that is,
their interiors should be removed and antipodal points in each resulting
boundary component identified, i.e., the depicted curve containing the asterisk bounds a M\"obius band in the surface. 

Though so natural, these new results are however non-existent in the literature and are apparently and surprisingly heretofore unknown.  In contrast, the existence of a sphere of projective measured foliations compactifying the Teichm\"uller space of a non-orientable surface without boundary to a closed piecewise-linear ball (Theorem \ref{thm:TBp}) has several times been stated without proof in the literature.   Moreover, on the level of 
homeomorphism classes of pants decompositions required for 1+1 dimensional topological quantum field theory with possibly non-orientable
surfaces as cobordisms, the corresponding weaker transitivity result was proved in \cite{AN}.

Holes in a surface may either be taken as boundary components or as punctures, and we treat the former in the body of the text and relegate the latter
to appendix \ref{app:zero}.  We begin by reviewing standard notions and results for non-orientable surfaces that differ from their analogues in the orientable case.
\begin{figure}%[hbt]%
 \begin{center}
 \includegraphics[width=10cm]{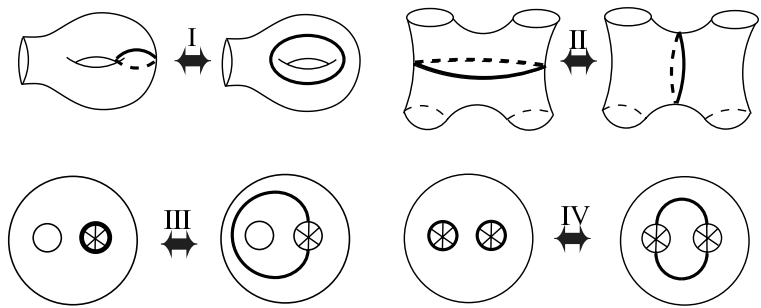}
 \caption{The four elementary moves. 
In each case, one or more of the depicted boundary components
may bound a M\"obius band both before and after the move.  In Move III, a 1-sided curve is replaced by another 1-sided curve, and in Move IV, two 1-sided curves are replaced by a 2-sided curve or conversely. 
%In Move V which is special to the connected sum of three projective planes, three 1-sided curves are replaced by one 1-sided and one 2-sided curve or conversely.
}
 \label{fig:elementary}%
\end{center}
\end{figure}

First, recall that in a non-orientable surface $N$ there is a distinction between 1-sided and 2-sided simple closed curves (depending respectively on whether
a regular neighborhood of the curve is a M\"obius band or annulus) and a further distinction between primitive and non-primitive 
2-sided curves (as elements of the fundamental group of $N$). A 2-sided simple closed curve is non-primitive if and only if it bounds a M\"obius band embedded in $N$, whose core is of course a 1-sided curve.

On the non-orientable surface $N$, there is a notion of Dehn twist along a 2-sided curve $c$  defined as in the orientable setting, up to the choice of a normal direction along the curve.  Namely, 
if $c$ is a 2-sided curve in $N$, then its annular neighborhood
in $N$ supports Dehn twists as usual; the difference is that we can only
determine the sense, right or left, of the Dehn twist relative to a specified normal direction to $N$ at a particular point on $c$ and so must
draw arrows to indicate the sense of a Dehn twist.
It is easy to see that a Dehn twist along a non-primitive 2-sided curve is isotopic rel boundary to the identity mapping: Indeed, consider the image of an arc decomposing the corresponding
M\"obius band into a rectangle; since this arc is found to be invariant under the Dehn twist up to proper homotopy, the result then follows from Alexander's trick.

\begin{figure}%[hbt]%
 \begin{center}
 \includegraphics[height=8cm]{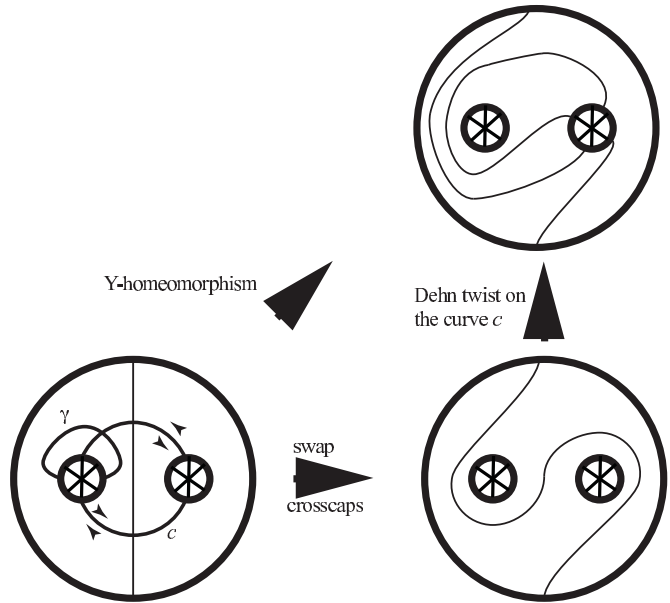}
 \caption{The Y-homeomorphism is supported on a Klein bottle minus a disk.  Note that the 1-sided curve $\gamma$ occurs in a neighborhood of the cross cap 
 connecting antipodal points.}
 \label{fig:Yhomeo}%
\end{center}
\end{figure}

In addition to Dehn twists, one must consider yet  another class of basic homeomorphisms \cite{Lickorish65} in the non-orientable case as follows.
Suppose that
$\gamma$ is a 1-sided and $c$ a 2-sided curve in $N$ so that $\gamma$ and $c$ meet transversely in a single point.
Let $K_1\subseteq N$ be a neighborhood of $\gamma\cup c$, which is homeomorphic to a Klein bottle
minus a disk, and let $M\subset K_1$ denote a neighborhood of $\gamma$, which is homeomorphic to a M\"obius band.
Finally, define the {\it Y-homeomorphism} $Y=Y_{\gamma, c}:N\to N$ to be the result of pushing $M$ once along $\gamma$ keeping pointwise fixed
the boundary of $K_1$ and extending by the identity on $N-K_1$; Figure \ref{fig:Yhomeo} illustrates the effect of $Y$ on an arc properly embedded in $K_1$.
 Put another way, $Y$ is the composition of the homeomorphism interchanging the two
cross caps followed by the Dehn twist along the curve $c$ as is also illustrated in the figure.  There are thus two Y-homeomorphisms
depending on the sense of the twisting, each one squaring to a distinct Dehn twist along the 2-sided curve $\partial K_1$.

Lickorish showed in \cite{Lickorish65} that the mapping class group (i.e., the group of homotopy classes of self-homeomorphisms\footnote{
In contrast, Lickorish called this the {\sl homeotopy} group of the surface reserving the term {\sl mapping class group}
for homotopy classes of orientation-preserving homeomorphisms in the orientable case.}) of a non-orientable surface is not generated by Dehn twists alone, but that any mapping class is isotopic to a composition of Dehn twists along a set of simple closed curves followed by a $Y$-homeomorphism.  Chillingworth, improving the result of Lickorish, exhibited a finite set of Dehn twists and $Y$-homeomorphisms which suffice \cite{Chillingworth1969}. In fact, Chillingworth's generators consist of a finite number of Dehn twists and a single additional $Y$-homeomorphism. 
%In \cite{Korkmaz2002}, Korkmaz treats the cases of non-orientable surfaces with punctures.  
Szepietowski \cite{Szep} gives  a minimal generating set in the non-orientable case as Humphries \cite{Humphries} had done in the orientable case.

Scharlemann (\cite{Scharlemann}, Theorem 1.1) showed that on a non-orientable surface, the set of 1-sided curves is isolated in the space ${\mathcal{PMF}}_0$ of projective measured foliations of compact support.  This implies in particular that simple closed curves are not dense in ${\mathcal{PMF}}_0$.  Danthony and Nogeuira showed that in ${\mathcal{PMF}}_0$, almost all measured foliations have a compact leaf which is 1-sided \cite{DN}. This contrasts with the case of orientable surfaces where almost all measured foliations are uniquely ergodic with dense leaves.

Korkmaz computed the first homology group of a closed non-orientable surface \cite{Korkmaz1998}; his result is analogous to that of Powell for orientable surfaces (who showed that for genus at least three the group is trivial) \cite{Powell1978}, completed in genus two by Mumford \cite{Mumford1967}.

%Atalan-Korkmaz

This paper is organized as follows.  Section 2 contains background material sufficient to precisely state our three main results therein.  Sections 3 and 4 are respectively dedicated to the generalized Fenchel-Nielsen and Dehn-Thurston Theorems.  Section 5 provides an analysis of pants decomposition of the Klein bottle minus a disk such as is necessary
to complete the proof of the generalized Hatcher-Thurston Theorem in section 6.  Section 7 finally contains closing remarks.

The proof that the space of projectivized measured foliations in a non-orientable surface with boundary itself provides a boundary of Teichm\"uller space as in the orientable case is sketched in appendix \ref{app:a}.  Appendix \ref{app:zero} treats each boundary component as a puncture in effect dropping twisting parameters and distinguished points
(and for those in ${\mathcal{MF}}_0$ that have compact support, also dropping the intersection numbers) for puncture-parallel pants curves.

In our exposition, we shall sometimes be brief in presenting arguments which in the non-orientable case follow closely the proofs in the orientable case,
however, we shall present detailed proofs when there are significant differences.  We shall also give details  concerning metrics on and measured foliations in surfaces with boundary that are not usually considered.

We would like to thank \"Oyk\"u Yurtta\c s for valuable discussions on this work. A sequel \cite{PPY} to this paper  is in collaboration with her.
Let us also gratefully acknowledge helpful input from Sergey Natanzon.

\section{Background and Statements of Results}\label{s:background}

In this section, we present current generalized versions of results that are well-known for
a connected orientable surface $F_{g,r}$. Our generalization concerns a non-orientable surface  $N_{g,r}$,
of genus $g\geq 0$ with $r\geq 0$
smooth boundary components where $2-2g-r<0$, respectively, $2-g-r<0$. 
Recall that the \emph{genus} of a non-orientable surface is defined, as in the orientable case, as the maximum number of simple closed curves  whose complement is connected;  any closed non-orientable surface of genus $g$ can be obtained as a connected sum of $g$ projective planes, whose Euler characteristic  is given by $2-g$.
We may sometimes say simply that a surface is {\it bordered} if it has boundary, i.e., if $r\geq 1$.

A \emph{pair of pants} is a surface homeomorphic to the sphere with the interiors of 3 disjoint closed disks removed. A {\it pants decomposition} ${\mathcal P}$ of an orientable or non-orientable surface is the (isotopy class) of a family of disjointly embedded simple closed curves whose complementary components are pairs of pants.  In particular, each boundary component occurs in any pants decomposition.  Of course, each pants curve in an orientable surface is 2-sided, but in a non-orientable surface, some may be 1-sided.  
In any case, we set
$\mathcal{P}=\mathcal{P}_1\cup \mathcal{P}_2$, where $\mathcal{P}_1$ is the set of 1-sided curves and $\mathcal{P}_2$ is the set of 2-sided curves. We consider the boundary curves of the surface as 2-sided (although they only have ``one side" in the surface).

Embedded non-primitive curves never occur in a pants decomposition since by definition they bound a M\"obius band.
Note that in the orientable case $S=F_{g,r}$, we have $\#{\mathcal P}_1=0$ with $\#{\mathcal P}_2=3g-3+2r$.  In contrast in the non-orientable case $S=N_{g,r}$, we can achieve a maximum of $\#{\mathcal P}_1=g$ with $\#{\mathcal P}_2=g+r-3$, and if there is a pants decomposition ${\mathcal P}$
with $\#{\mathcal P}_i=p_i$, for $i=1,2$ and $p_1\geq 2$, then there is also a pants decomposition with ($p_1-2$) 1-sided and ($p_2+1$) 2-sided curves;
in particular, if $g=2k$ is even, we can achieve $\#{\mathcal P}_1=0$ in the non-orientable case as well with $\#{\mathcal P}_2=g+r+k-3$.  

We shall rely heavily on geodesic representatives of simple closed curves in hyperbolic surfaces,  by which we mean complete, finite-area metrics with constant Gaussian curvature -1 and geodesic
boundary. For this matter, there are small differences between the orientable and the non-orientable case. If $S=F_{g,r}$ or  $N_{g,r}$ is equipped with a hyperbolic structure, then  any simple closed curve $c$ on $S$ has a unique geodesic representative which is simple except in the case where $c$ is 2-sided non-primitive; in that case, its geodesic representative is not embedded in $S$, but it traverses two times the geodesic representative of the core curve of the M\"obius band that $c$ bounds. In the orientable as well as in the non-orientable case, the geodesic representatives of two disjoint simple closed curves are themselves disjoint except for the special case that we just described. More generally, if two simple closed curves $c_1$ and $c_2$ on $S$ are such that there is no disk embedded in $S$ whose boundary consists of the union of an arc contained in $c_1$ with an arc contained in $c_2$, and if $c'_1$ and $c'_2$ are the geodesics in the homotopy classes of $c_1$ and $c_2$, respectively, then the union $c_1\cup c_2$ is isotopic to the union $c'_1\cup c'_2$. In particular, the intersection number $i(c_1,c_2)$ is equal to the number of points of intersection between the geodesic curves $c_1'$ and $c_2'$.

A fundamental result in low-dimensional topology is the {\it Hatcher-Thurston Theorem} \cite{HT} which uses Cerf theory to prove
that any two pants decompositions of an orientable surface are related by a finite sequence of the
the first two {\it elementary moves} illustrated in Figure \ref{fig:elementary}  (Moves I and II).
Our generalization Theorem \ref{thm:ght} that includes the orientable case is:

\begin{theorem}[Generalized Hatcher-Thurston]
Any two pants decompositions of a possibly non-orientable and possibly bordered compact surface are related by a finite sequence of the {\sl elementary            moves} (Moves I to IV)
illustrated in Figure~\ref{fig:elementary}. 
\end{theorem}
 
The Fenchel-Nielsen and Dehn-Thurston twisting numbers of curves in a pants decomposition ${\mathcal P}$ of $S$ that we employ in Theorems \ref{thm:noFN} and \ref{thm:noDT} below require us, as in the orientable case, to make choices, namely, we shall choose an orientation on each pants curve, together with a distinguished point on that curve, and for each pair of pants in the decomposition, we shall associate to each boundary component (equipped with its orientation and its distinguished point) a homotopy class of arcs joining the distinguished point on that component to the distinguished point of another boundary component of the pair of pants. The definition of the Fenchel-Nielsen and Dehn-Thurston twisting numbers rely on these data. In the case of an oriented surface, we get a new way of measuring twisting numbers which does not coincide with the classical one (e.g., the one explained in \cite{Thurston}), but of course it works as well\footnote{
In fact, Wolpert's formula for the Weil-Petersson K\"ahler two-form in the 
orientable case also applies here with the new twists according to \cite{FN2}.} for providing twist parameters.

In particular, each boundary component of $\partial S$ (if any) of a possibly non-orientable surface $S$ comes decorated with a distinguished point.
If this boundary $\partial S$ is comprised of 2-sided curves $\partial S=\partial_1\sqcup \cdots \sqcup \partial_m$, equipped with a collection of  {\it distinguished points} $p_j\in \partial _j$, one in each component of the boundary, for $j=1,\ldots, m$, then we let
$\vec p$ denote the collection of these points indexed by components $j=1,\ldots ,m$ of the boundary.  

To study hyperbolic structures and measured foliations on surfaces $S$ with boundary and the mapping class group actions on spaces of equivalence classes of such structures,
we must make certain assumptions about isotopies of homeomorphisms of $S$. One thing that distinguishes {\sl bordered} surfaces is that the isotopies that define the equivalence relations between hyperbolic metrics and between measured foliations, as well as the homeomorphisms that represent mapping classes, must fix the distinguished points $\cup \vec p$, either setwise in general or pointwise in the ``pure'' case.
In particular, a Dehn twist along a simple closed curve parallel to a boundary component of the surface will not be considered as isotopic to the identity. This will be particularly useful when the surface is divided into pairs of pants where separate structures in individual pairs of pants are glued together.

Suppose that $\rho$ is a hyperbolic metric on $S$. Each free homotopy class of closed curve admits a unique $\rho$-geodesic, except that, as we already noted, in the non-orientable case,
the geodesic representative of an embedded non-primitive curve may fail to be embedded. 
In particular, any isotopy class of closed curve has its $\rho$-geodesic length.

Following  \cite{Penner-bordered}, the {\it Teichm\"uller space} ${\mathcal T}(S)$ of a possibly bordered and possibly non-orientable surface $S=F_{g,r}$ or $S=N_{g,r}$ is the quotient space
of all pairs $(\rho,\vec p)$, where $\rho$ is a hyperbolic metric on $S$  and $\vec p$ is the indexed collection of distinguished points, one in each oriented boundary component.
The quotient is taken by simultaneous push-forward of metric and distinguished points under  diffeomorphisms $f:S\to S$  isotopic to the identity, i.e.,
$(\rho,\vec p)~\sim~(f_*(\rho),f(\vec p))$, where the isotopy of $f$ to the identity must fix the boundary pointwise.

The classical Fenchel-Nielsen Theorem \cite{FN,FN2} extended to bordered surfaces parametrizes the Teichm\"uller space ${\mathcal T}(F)$ of an orientable surface $F=F_{g,r}$ by assigning to each curve $\{ c_j\} _{j=1}^{3g-3+2r}$ in a pants decomposition ${\mathcal P}={\mathcal P}_2$ of $F$ the $\rho$-length $\ell _j(\rho)>0$ of $c_j$, where $\rho$ is a hyperbolic metric on $S$, as well as another
so-called twisting number $\theta_j(\rho)\in{\mathbb R}$, whose standard definition \cite{FN,FN2} is modified in the next section \ref{sec:gfn}.
Our generalization that includes the orientable case is:

\begin{theorem}[Generalized Fenchel-Nielsen] \label{thm:noFN}
Fix a pants decomposition ${\mathcal P}={\mathcal P}_1\cup {\mathcal P}_2$ of $S=F_{g,r}$ or $S=N_{g,r}$, where ${\mathcal P}_1=\{b_1,\ldots ,b_M\}$ and
${\mathcal P}_2=\{ c_1,\ldots ,c_N\}$. Given a hyperbolic metric $\rho$ on
$S$, let $\mu_j(\rho)$ denote the $\rho$-length of $b_j$, for $j=1,\ldots ,M$, and $\ell_i(\rho)$ and $\theta_i(\rho)$ as before, respectively, denote $\rho$-lengths and twisting numbers of $c_i$, for $i=1,\ldots ,N$.
Then the mapping
\begin{equation}\label{eq:GFN}
\aligned
{\mathcal T}(S)&\to  {\mathbb R}_{>0}^M\times ({\mathbb R}_{>0}\times {\mathbb R})^N\cr
\rho&\mapsto (\mu_1(\rho),\ldots ,\mu_M(\rho), \ell_1(\rho),\theta_1(\rho),\ldots, \ell_N(\rho),\theta_N(\rho))\cr
\endaligned
\end{equation}
is a real-analytic homemorphism.
\end{theorem}
 
 \begin{corollary} The moduli space ${\mathcal M}(N)={\mathcal T}(N)/MC(N)$ of a non-orientable surface $N=N_{g,r}$ is itself non-orientable
provided $g\geq 2$ and $2-g-r<0$.
\end{corollary}
\begin{proof} 
Choose two disjoint 1-sided curves $c_1,c_2$ and an embedded arc $\alpha$ connecting them and let $c_3$ be the 2-sided curve
isotopic to the boundary of a regular neighborhood of $c_1\cup c_2\cup\alpha$.  Extend $c_1,c_2,c_3$ to a pants decomposition of $N$.
Thus, $c_3$ bounds a M\"obius band minus a disk
containing $c_1,c_2$ and supports the Y-homeomorphism $Y_{c_3}$, whose effect on generalized Fenchel-Nielsen coordinates 
is to interchange the length coordinates of $c_1$ and $c_2$ and to add to the twist coordinate of $c_3$ half its length coordinate
leaving all other generalized Fenchel-Nielsen coordinates unchanged.  This transformation evidently reverses the orientation
of ${\mathcal T}(N)$ as required.
\end{proof}

We define ${\mathcal D}={\mathcal D}(S)$ to be the collection of all homotopy classes of  1-submanifolds properly embedded in $S-\cup\vec p$,
no curve component of which is inessential and no arc component of which is boundary parallel, modulo ambient isotopy in $S$ pointwise fixing $\cup\vec p$.  Let $\Pi{\mathcal D}=\Pi{\mathcal D}(S)\subseteq {\mathcal D}(S)$ denote the subset of homotopy classes of primitive elements.  

We denote by $\mathcal{C}$ (respectively, $\Pi \mathcal{C}$) the subset of homotopy classes of connected elements of $\mathcal{D}$ (respectively, $\Pi \mathcal{D}$). The elements of  $\mathcal{C}$ are homotopy classes of simple closed curves or proper arcs. 
Finally, we denote by $\mathcal{S}$ (respectively, $\Pi \mathcal{S}$) the subset of $\mathcal{C}$ (respectively, $\Pi \mathcal{C}$) of elements which consists of simple closed curves.

There is an inclusion map
\begin{equation} \label{eq:inclusion1}
{\mathcal T}(S)\hookrightarrow \mathbb{R}_{>0}^{\Pi{\mathcal C}}
\end{equation}
sending each hyperbolic structure to the map $\Pi{\mathcal C}\to \mathbb{R}_{>0}$ which assigns to each homotopy class of simple closed curve or arc the length of its geodesic representative.
As in the case of orientable surfaces, the topology of ${\mathcal T}(S)$ can be induced from that of the space $\mathbb{R}_{>0}^{\Pi{\mathcal C}}$ via the inclusion map (\ref{eq:inclusion1}). By passing to the homothetic quotient of $\mathbb{R}_{>0}^{\Pi{\mathcal C}}$ by multiplication by positive reals, we get a map into the projective space:
\begin{equation} \label{eq:inclusion2}
{\mathcal T}(S)\hookrightarrow P\mathbb{R}_{>0}^{\Pi{\mathcal C}}.
\end{equation}

More precisely, we have the following:
\begin{theorem}\label{th:inclusion1}
 The inclusion map (\ref{eq:inclusion2}) is a homeomorphism onto its image.
\end{theorem}

Theorem \ref{th:inclusion1} follows from Theorem \ref{thm:noFN}, as in the orientable case, by encoding the twist parameter along a 2-sided curve in a pants decomposition by the length parameters of a pair of curves that are contained in the pants (or pant) adjacent to that curve in the surface.

Let ${\mathcal {MF}}(S)$ denote the space of measured foliations of $S$ modulo isotopy rel $\cup\vec p$ and Whitehead moves, cf.\ \cite{FLP}. 
As is well-known in the orientable case $F=F_{g,r}$, we may regard 
${\mathcal S}(F)=\Pi{\mathcal S}(F)\subset {\mathcal {MF}}(F)$
 by ``enlarging'' curves and arcs with counting measure to measured foliations.
In contrast, in the non-orientable case, the core and the boundary of an embedded M\"obius band have the same
enlargements, so it is only the primitive curves that embed for $N=N_{g,r}$, i.e., 
$\Pi{\mathcal S}(N)\subset {\mathcal {MF}}(N)$.

 \begin{theorem}[Dehn-Thurston \cite{Penner-thesis,Harer-Penner}] \label{thm:DT}
Fix a pants decomposition ${\mathcal P}={\mathcal P}_2=\{ c_1,\ldots ,c_N\}$ of $F=F_{g,r}$, where $N=3g-3+2r$. 
Given a measured foliation $\mathcal F$ on $F$, let $m_i(\mathcal F)\geq 0$ denote the $\mathcal F$-transverse measure,
which is also called the {\sl intersection number}, 
of $c_i$, for $i=1,\ldots ,N$.  Further define real
{\sl twisting numbers} $t_i (\mathcal F)\geq 0$, for $i=1,\ldots ,N$,
cf.\ \cite{Penner-thesis,Harer-Penner}, where $t_i(\mathcal F)\geq 0$ if $m_i(\mathcal F)=0$.  We shall recall the definition of these parameters in section \ref{s:Dehn-Thurston} below. Then the mapping
$$\aligned
{\mathcal {MF}}(F)&\to (({\mathbb R}_{>0}\times{\mathbb R})\cup (\{0\} \times {\mathbb R}_{\geq 0}))^N\cr
\mathcal F&\mapsto (m_1(\mathcal F),t_1(\mathcal F),\ldots ,m_N(\mathcal F),t_N(\mathcal F))\cr
\endaligned$$
is a bijection.  
\end{theorem}

The condition $t_i(\mathcal F)\geq 0$ if $m_i(\mathcal F)=0$ is explained when we present the generalized Dehn-Thurston theorem in \S \ref{s:Dehn-Thurston} below.

\begin{theorem}[Generalized Dehn-Thurston] \label{thm:noDT}
Fix a pants decomposition ${\mathcal P}={\mathcal P}_1\cup {\mathcal P}_2$ of $S=F_{g,r}$ or $S=N_{g,r}$, where ${\mathcal P}_1=\{b_1,\ldots ,b_M\}$ and
${\mathcal P}_2=\{ c_1,\ldots ,c_N\}$.  
Given a measured foliation $\mathcal F$ in $F$, define real parameters $n_j(\mathcal F)$ as follows:
if the $\mathcal F$-transverse measure of $b_j$ is positive, then $n_j(\mathcal F)\geq 0$ agrees with this intersection number;
if $\mathcal F$ contains a closed leaf isotopic to $b_j$, then $n_j(\mathcal F)\leq 0$ is taken to be the negative of the maximum width of the
band of leaves foliating a M\"obius band neighborhood of $b_j$,
for $j=1,\ldots ,M$.  Let $m_i(\mathcal F)$ and $t_i(\mathcal F)$ denote the intersection and twisting numbers of $\mathcal F$ on $c_i$, for $i=1,\ldots ,N$ as in Theorem \ref{thm:DT}, and which we recall in \S \ref{s:Dehn-Thurston} below.
Then the mapping
$$
\aligned
{\mathcal{J}: \mathcal {MF}}(S)&\to {\mathbb R}^M\times (({\mathbb R}_{>0}\times{\mathbb R})\cup (\{0\} \times {\mathbb R}_{\geq 0}))^N\cr
\mathcal F&\mapsto (n_1(\mathcal F),\ldots, n_M(\mathcal F),m_1(\mathcal F),t_1(\mathcal F),\ldots ,m_N(\mathcal F),t_N(\mathcal F))\cr
\endaligned
$$
is a bijection.   
\end{theorem}

We use the bijection in Theorem \ref{thm:noDT} to define the topology of $\mathcal{MF}$, and to do so we must be more precise on how the two spaces $F_1= {\mathbb R}_{>0}\times{\mathbb R}$ and $F_2= \{0\} \times {\mathbb R}_{\geq 0}$
 in the above set-theoretic union are glued together. For this, we extend the space $F_1= {\mathbb R}_{>0}\times{\mathbb R}$ to $F_1= {\mathbb R}_{\geq 0}\times{\mathbb R}$ by adding a copy of $\mathbb{R}$, and we glue that copy of $\mathbb{R}$ to the first factor of the product $\mathbb{R}\times F_2$ by the identity map; this 
 has the effect that the asymptotes for right and left twisting coincide.
 
 \begin{theorem} [Generalized Thurston Boundary for possibly bordered surfaces]\label{thm;TB}
 Fix a pants decomposition ${\mathcal P}$ of a possibly non-orientable and possibly bordered surface $S$.
 Dehn-Thurston coordinates with respect to ${\mathcal P}$ establish a homeomorphism $\mathcal{MF}(S)\approx {\mathbb R}^{(\#{\mathcal P}_1+2\#{\mathcal P}_2)}$.
 The corresponding projectivized space ${\mathcal{PF}}(S)$ is thus a sphere
 of dimension $\#{\mathcal P}_1+2\#{\mathcal P}_2-1$, which has a natural piecewise-linear structure independent of the pants decomposition
 and provides a boundary $\overline{\mathcal{T}}(S)={\mathcal T}(S)\cup {\mathcal{PF}}(S)$ to the open ball that is Teichm\"uller space ${\mathcal T}(S)$ so as to form a closed ball
 $\overline{\mathcal{T}}(S)$.  The usual
 action of the mapping class group of $S$ on ${\mathcal T}(S)$ extends continuously to $\overline{\mathcal T}(S)$ by its natural action on ${\mathcal{PF}}(S)$.
 \end{theorem}
 
 \begin{proof}
 As just discussed, each 2-sided curve in a pants decomposition accounts for a factor ${\mathbb R}^2$ in $\mathcal{MF}$, and in fact,
 $\mathcal{MF}(S)\approx {\mathbb R}^{(\#{\mathcal P}_1+2\#{\mathcal P}_2)}$  
 since each element of ${\mathcal P}_1$ likewise contributes a factor ${\mathbb R}$ according to the generalized Dehn-Thurston Theorem.
 It thus follows as a corollary that ${\mathcal {PF}}(S)$ is a sphere of the asserted dimension from the generalized Dehn-Thurston Theorem.  The argument that this sphere provides a boundary compactifying ${\mathcal T}(S)$ to a closed ball with the asserted mapping class properties is given in appendix \ref{app:a} and is thus separated from the discussion here only because of its very substantial reliance on the classical theory.  The proof that the piecewise-linear structure of the sphere ${\mathcal {PF}}(S)$  is well-defined is lengthy but follows the same lines  as the corresponding result \cite{Penner-thesis}  for orientable surfaces.
 \end{proof}

In the proof of Theorem \ref{thm;TB},  mimicking the orientable case, we must study the inclusion map
\begin{equation} \label{eq:inclusion3}
{\mathcal {MF}}(S)\hookrightarrow \mathbb{R}_{\geq 0}^{\Pi{\mathcal C}}
\end{equation}
sending each measured foliation to the map $\Pi{\mathcal C}\to \mathbb{R}_{\geq 0}$ which assigns to each homotopy class of simple closed curve or proper arc its intersection number with the given foliation.
As in the case of orientable surfaces, the topology of ${\mathcal {MF}}(S)$ in the non-orientable case can be induced from that of the space $\mathbb{R}_{\geq 0}^{\Pi{\mathcal C}}$ via the inclusion map (\ref{eq:inclusion3}). 

Taking quotients by homothety throughout, we get maps between projective spaces
\begin{equation}\label{eq:inclusion4}
{\mathcal {PF}}(S)\hookrightarrow P\mathbb{R}_{\geq 0}^{\Pi{\mathcal S}}
\end{equation}
and have the following theorem, the analogue for measured foliations of Theorem \ref{th:inclusion1} for hyperbolic structures.
\begin{theorem}\label{th:inclusion2}
 The inclusion map (\ref{eq:inclusion4}) is a homeomorphism onto its image.
\end{theorem}

The proof also follows the same lines as the proof of the analogous result in the orientable case. It makes use of Theorem
\ref{thm:noDT}, by encoding the twist parameter of a foliation along a 2-sided curve in a pants decomposition by the transverse measure parameters of a pair of curves that are contained in the pants (or pant) adjacent to that curve in the surface. (Recall that there are no twist parameters along a 1-sided curve.)

 \section{Generalized Fenchel-Nielsen Theorem}\label{sec:gfn}
 
 Recall that we have chosen a collection of distinguished points $\vec p$, one such point lying in each boundary component,
 as well as specifying an orientation on each component of the boundary.

Let ${\mathcal P}={\mathcal P}_1\cup {\mathcal P}_2$ be a pants decomposition  of $S=F_{g,r}$ or $S=N_{g,r}$, where ${\mathcal P}_1=\{b_1,\ldots ,b_M\}$ and
${\mathcal P}_2=\{ c_1,\ldots ,c_N\}$ are as in the statement of Theorem \ref{thm:noFN}. For each hyperbolic metric $\rho$ on $S$, each of the curves $\{b_1,\ldots ,b_M\}$ and $\{ c_1,\ldots ,c_N\}$ has a unique $\rho$-geodesic representative. To spare notation, we use the same letter $S$ for the surface $S$ equipped with a hyperbolic structure $\rho$, and $b_i$ and $c_i$ for the $\rho$-geodesic curves in the homotopy classes $b_i$ and $c_i$. A geodesic curve $C=b_i$ or $c_i$ thus has a well-defined \emph{length parameter} $\ell_C(\rho)\in \mathbb{R}_{>0}$. It also has a \emph{twist parameter}
$\theta_C(\rho)\in{\mathbb R}$, which is defined modulo conventions that we shall state precisely, and which measures the relative twist amount along $C$ between the two pairs of pants that have this geodesic in common (where the two pairs of pants may coincide). 

In the classical setting, the definition of the twist parameter is usually given in the case when the curve $C$ is not a boundary curve of the surface, see, e.g.,  \cite[Theorem 4.6.23]{Thurston}, but we shall need to define the twist parameter on $C$ including the case where $C$ is a boundary geodesic of the surface, and we begin with the definition in this case. 

Let us equip each  curve in $\mathcal{P}$ with an orientation and with a distinguished point; if the curve is a boundary component of the surface $S$, then these
data are the ones we have already specified including  $\vec p$. 

If $C$ is a boundary curve of the surface, we choose an isotopy class rel boundary of an arc $\alpha$ connecting the distinguished point from $\cup\vec p$ in $C$ to the distinguished point boundary component of another chosen component of the pair of pants. Now, when the pants decomposition is made geodesic, 
there is an isotopy moving the boundary that carries $\alpha$ to a geodesic arc perpendicular to the two boundary components that it connects.  During this isotopy,
the endpoint of $\alpha$ in $C$ passes a certain number $n$ of times through the distinguished point $p$ it contains;  $n$ comes with a sign, where a ``positive crossing"
corresponds to the specified orientation on $C$.  We let $\ell$ denote the hyperbolic length of $C$
and $\tau$ the hyperbolic length along $C$ from $p$ to the endpoint of the perpendicular arc in the specified orientation.  The {\it twisting number} of the configuration
$(p,\alpha)$ is defined to be $t(p,\alpha)=\tau + n \ell$.
%
%  To define the  twisting numbers 
%for the boundary components  of the surface, it remains to determine a sign on the pre-twisting numbers, 
% and to this end, we must still compare the right or left sense of the pre-twisting with that specified by the chosen orientation on a neighborhood of the boundary.
%
%
%
%If $C$ is an interior 2-sided pants curve, then we choose two perpendicular arcs $\alpha$ and $\beta$ with endpoints at the distinguished point $p$ on $C$, coming from adjacent pairs of pants, as well as an orientation on $C$ in order to define the pre-twisting number for $C$ with respect to $\rho$ as
%$t(p,\alpha)-t(p,\beta)$.  In order to associate a sign to this pre-twisting number and finally define the twisting number itself, we must again compare 
%the direction of twisting in the surface to the orientation that we chose {\sl a priori} on neighborhoods of the 2-sided pants curves.  (Perhaps it is worth mentioning that this description of Fenchel-Nielsen twisting numbers is a new point of view even in the classical case of closed orientable surfaces.)
% 

Let $C$ now be an interior 2-sided pants curve. We choose, as for the case where $C$ is a boundary curve of the surface, an isotopy class rel boundary of an arc $\alpha$ connecting the distinguished point on $C$ to the one of another boundary component of  the pair of pants. Again, when the pants decomposition is made geodesic, 
there is an isotopy moving the boundary that carries $\alpha$ to a geodesic arc perpendicular to the two boundary components that it connects, and during this isotopy,
the endpoint of $\alpha$ in $C$ passes a certain number $n$ of times through the distinguished point it contains; the integer $n$ comes with a sign, where a ``positive crossing"
corresponds to the specified orientation on $C$. Letting as before $\ell$ denote the hyperbolic length of $C$
and $\tau$ the hyperbolic length along $C$ from the distinguished point to the endpoint of the perpendicular arc in the specified orientation, the {\it pre-twisting number} of the distinguished point in $C$ is defined to be $t(p,\alpha)=\tau + n \ell$. Note that the sign of $\tau$, like that of $n$, is positive or negative according to whether the motion is in the sense of the orientation of the curve or in the opposite sense. The {\it twisting number} of the configuration is the difference  of the two pre-twisting numbers  coming from adjacent pairs of pants; the determination of which side occurs with which sign in taking the difference can be made from an orientation on a neighborhood of $C$ in $S$ 
(as induced by inclusion when $S$ is oriented) together with the orientation on $C$ itself.

Now we prove Theorem \ref{thm:noFN}.

We cut the surface $S$ along the geodesic curves in the pants decomposition $\mathcal{P}$ to obtain a collection $\{P_i\}$ of hyperbolic pairs of pants with geodesic boundary. The isometry type of the hyperbolic structure on each of the pairs of pants is completely determined by the lengths of its boundary components, and there is an isometric identification between boundary components of these pairs of pants which is determined by the fact that they come from the hyperbolic structure $\rho$ on $S$. In this correspondence, each curve arising from the set $\mathcal{P}_2$ is identified with another curve arising from the same set, and each curve arising from a curve in $\mathcal{P}_1$ is identified to itself. More precisely, the surface $S$ can be recovered from the collection $\{P_i\}$ in the following way:

$\bullet$ For each boundary component of a hyperbolic pair of pants $P$ which is in $\mathcal{P}_1$, glue it to itself by sending each point to its diametrically opposite point, with respect to any parametrization of this component by arc length.

$\bullet$ For each boundary component of a hyperbolic pair of pants $P$ which is in $\mathcal{P}_2$, glue it to the corresponding boundary component using the twist parameter determined by the hyperbolic structure $\rho$ and the convention used to measure the twists.

For any hyperbolic metric $\rho$ on $S$ and for any element $b_j$ of $\mathcal{P}_1$, we denote by $\mu_j(\rho)$ the $\rho$-length of $b_j$, for $j=1,\ldots ,M$. For an element $c_i$ in $\mathcal{P}_2$, we denote by $\ell_i(\rho)$ and $\theta_i(\rho)$ the $\rho$-lengths and twisting numbers of $c_i$, for $i=1,\ldots ,N$. We show that the mapping
\begin{equation} \label{eq:FN1}
\aligned
{\mathcal T}(S)&\to  {\mathbb R}_{>0}^M\times ({\mathbb R}_{>0}\times {\mathbb R})^N\cr
\rho&\mapsto (\mu_1(\rho),\ldots \mu_M(\rho), \ell_1(\rho),\theta_1(\rho),\ldots, \ell_N(\rho),\theta_N(\rho))\cr
\endaligned
\end{equation}
is injective.

Let $\rho$ and $\sigma$ be two hyperbolic metrics on $S$.
If $\rho$ and $\sigma$ assign different lengths to any one of the 1-sided or 2-sided curves in $\mathcal{P}$, then the two metrics cannot be isotopic, by the uniqueness of the geodesic length in each homotopy class of simple closed curves on $S$.

Assume now the length parameters are all equal, and assume that the $\rho$- and $\sigma$-twist parameters along some 2-sided curve are different. Consider the two pairs of pants (which can be the same) which are adjacent to this curve. Since the curve is 2-sided, the situation is the same as the one studied in \cite{FLP}, and we can find a simple closed curve in the union of these two pairs of pants whose $\rho$- and $\sigma$-length parameters are distinct, showing that the metrics $\sigma$ and $\rho$ are not homotopic.

This shows that the map (\ref{eq:FN1}) is injective. It is also clearly surjective since the length and twist parameters uniquely determine the hyperbolic structures on the pairs of pants and the gluing homeomorphisms between adjacent ones.
The topology of  $\mathcal{T}(S)$ can be defined using these injective maps associated to pairs of pants. The fact that they are real-analytic is seen as in the classical case (cf.\ \cite{Abikoff}).
%
%Recall now that the topology of $\mathcal{T}(S)$ is defined through its embedding in the space $\mathbb{R}_+^{\Pi{\mathcal S}}$ (Theorem \ref{th:inclusion1}).

\section{Generalized Dehn-Thurston Theorem}\label{s:Dehn-Thurston}

 The Dehn-Thurston theorem provides a global parametrization of the space of homotopy classes of measured foliations, and it is used to give the topological type of this space. The proof involves the classification of measured foliations on pairs of pants with the twist parameters used to glue foliations on adjacent pairs of pants.
 
We start with facts on measured foliations on surfaces with boundary.  The local models of the singular points for measured foliations on non-orientable surfaces are the same as those on orientable surfaces. For points in the interior of the surface, these are generalized saddles ($n$-prong singularities, $n\geq 3$, see Figure \ref{singular}.) By definition, the model of a singular point on the boundary is such that if we double the surface along the boundary component that contains that boundary point then we have the usual model of a singular point at an interior point (see Figure \ref{double}).

\begin{figure}[hbt]%
 \begin{center}
 \includegraphics[width=12cm]{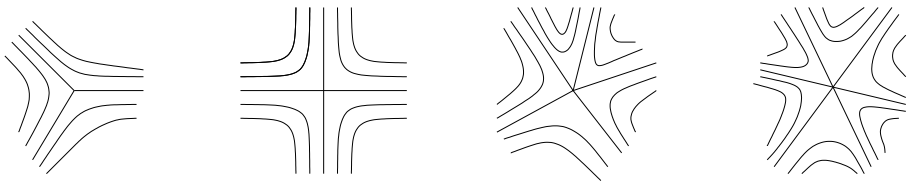}
 \caption{The local models for a singular point at an interior point of the surface.}
 \label{singular}%
\end{center}
\end{figure}

\begin{figure}[hbt]%
 \begin{center}
 \includegraphics[width=3.cm]{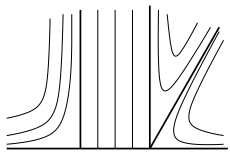}
 \caption{The model for a singular point on the boundary is such that it becomes a model for an interior singular point when we double the surface along that boundary component.}
 \label{double}%
\end{center}
\end{figure}

\begin{lemma} \label{lem:fol-Mob}
A measured foliation of a M\"obius band is of one of the following two types:
\begin{enumerate}
\item \label{fol1} all the leaves are simple closed curves, and one of them (the ``core curve") is 1-sided while the others are homotopically double covers of this 1-sided curve;
\item \label{fol2} all the leaves are arcs joining the boundary to itself.
\end{enumerate}
\end{lemma}

\begin{proof}
Consider an annulus which double covers the M\"obius band, equipped with the measured foliation induced from that covering. There are two cases:

Assume first that the boundary component of the annulus is a leaf. Doubling the annulus along its two boundary components, we get a torus equipped with measured foliation that has a closed leaf. Therefore, all the leaves of that foliation are closed and homotopic to each other. This shows that we are in case (\ref{fol1}) of the statement.

Assume now that the boundary component of the M\"obius band is not a leaf. By taking an annular double cover and then a torus double again, we obtain a foliation of the torus. Such a foliation cannot have singular points. Therefore, the foliation on the annulus is transverse to boundary curves of  the annulus, and any leaf is an arc joining a boundary component to the other one (this follows from the ``stability lemma" (see \cite{FLP} Lemma II.4 of Expos\'e 5). From this, it is easy to see that we are in case  (\ref{fol2}) of the statement.

\end{proof}
 
 To define the intersection and twisting coordinates for the 2-sided curves that appear in the statement of Theorem \ref{thm:noDT}, we need first recall the classification of equivalence classes of measured foliations on pairs of pants (see \cite{FLP}, Expos\'e 6, \S 2).
 
  Let $P$ be a pair of pants with boundary components $\partial_i$ ($i=1,2,3$) equipped with a measured foliation $\mathcal{F}$.  It also follows from the classification of measured foliations on pairs of pants (\cite{FLP} Expos\'e 6) that a boundary component of $P$ is then either transverse to $\mathcal{F}$ or it is a (perhaps singular) leaf; that is, there are no mixed cases as in Figure \ref{double}. In the latter case (that is, when $\mathcal{F}$ is not transverse to a boundary component), then we say that $\mathcal{F}$ is ``tangent'' to the given boundary component. 
  
  If $\mathcal{F}$ meets each boundary component of the pair of pants $P$  transversely, then the equivalence class of $\mathcal{F}$ is determined by the three positive numbers $i(\mathcal{F},\partial_i)$.
In the case where $\mathcal{F}$ is tangent to some boundary component $\partial_i$,  there is a maximal annulus having $\partial_i$ as a boundary component and  foliated by leaves which are homotopic to $\partial_i$. (The annulus may be degenerate, that is, reduced to a single curve, and in this case, the curve necessarily contains a singular point). We shall call the transverse measure of an arc joining the two boundary components of that annulus the \emph{twisting number} of that boundary component, for reasons that will become apparent below. In any case, the equivalence class of a measured foliation $\mathcal{F}$ on the pair of pants $P$ is determined by three nonnegative numbers, one number attached to each of the three boundary components, and this number is either a transverse measure (and in this case is positive) or it is a twisting number (and in the latter case is either positive or zero).

Let us fix a pants decomposition ${\mathcal P}$ of $F=F_{g,r}$ of $N_{g,r}$. It follows from \cite{FLP}, Expos\'e 6 \S 2 (and the proof carries over to the non-orientable case) that any equivalence class of measured foliation $\mathcal{F}$ on $F$ has a representative which is in \emph{good position} with respect to  ${\mathcal P}$, that is, this representative is a measured foliation (which we also denote by $\mathcal{F}$) such that each pants curve is either transverse to $\mathcal{F}$ or is a (perhaps singular) leaf of $\mathcal{F}$. In particular,  $\mathcal{F}$ induces on each pair of pants a measured foliation in the usual sense, and  the above classification and parameters of foliations on pairs of pants are available.  

We now prove Theorem \ref{thm:noDT}.
 
Let ${\mathcal P}={\mathcal P}_1\cup {\mathcal P}_2$ be a pants decomposition of $S=F_{g,r}$ or $S=N_{g,r}$, where ${\mathcal P}_1=\{b_1,\ldots ,b_M\}$ and
${\mathcal P}_2=\{ c_1,\ldots ,c_N\}$ are, as before, the 1-sided and the 2-sided curves respectively. Given an (equivalence class of) measured foliation  $\mathcal{F}$, we can assume that it is in good position with respect to $\mathcal{P}$. For such a measured foliation, the intersection number $i(\mathcal{F},C)$ of each pants curve  $C=b_i$ or $c_i$ is equal to the actual transverse $\mathcal{F}$-measure of $C$.

 Let us now define the intersection parameters of $C$.  
 In the case where $C=b_j$ is 1-sided, the intersection number $n_j(\mathcal{F})\in{\mathbb R}$ is defined as  in the statement of Theorem \ref{thm:noDT}, that is, if the foliation contains a closed leaf isotopic to $b_j$, then  $n_j(\mathcal{F})$ is  taken to be the negative of the maximum width of the band of leaves foliating a M\"obius band neighborhood of $b_j$, while $n_j({\mathcal F})$ is positive if and only if it agrees with $i({\mathcal F},b_j)\neq 0$.  There is no twisting number defined in this case.
 In the case where $C=c_j$ is 2-sided,  the intersection number $m_j(\mathcal{F})=i({\mathcal F},C)\geq 0$ is the $\mathcal{F}$-measure of $C$. We now define the twisting number $t_j(\mathcal{F})$. 
 First assume that $i(\mathcal{F}, c_j)=0$. Since $\mathcal{F}$ is good position with respect to $\mathcal{P}$, $c_j$ is a (perhaps singular) leaf of $\mathcal{F}$. The \emph{twisting number} $t_j(\mathcal F)$ of $\mathcal{F}$ with respect to $\mathcal{P}$ is then the transverse measure of a transverse arc 
joining the two boundary components of the maximal foliated annulus in $S$ parallel to $c_j$. (The maximal annulus might be reduced to the curve $c_j$ itself, and in this case the transverse measure is zero.)

In the case where $i(\mathcal{F},c_j)>0$, the twisting parameter is defined in a way analogous to the twisting number in the case of the Fenchel-Nielsen parameters: We equip, as in the case of the twisting parameters for hyperbolic structures,  the 2-sided pants curve $c_j$ with an orientation and with a distinguished point. As before, if the curve is a boundary component of the surface $S$, then the distinguished point and orientations are the ones that we already have specified. Again, for each pants curve $c_j$, we choose an isotopy class rel boundary of arc $\alpha$ connecting $c_j$ to another chosen boundary component of each pair of pants to which it belongs.  We can now define the twisting parameter on the pants curve $c_j$.

We first consider the case where $c_j$ is a boundary component of $F$. Since we assumed the measured foliation $\mathcal{F}$ to be in good position with respect to the pair of pants decomposition $\mathcal{P}$ and since $i(\mathcal{F},c_j)>0$, it must be that $c_j$ is transverse to $\mathcal{F}$ and there is an isotopy moving the points on this pants curve that carries $\alpha$ to an arc joining the two boundary components of the given pair of pants that it connects and which is in minimal position with respect to $\mathcal{F}$.  Note that this arc necessarily arrives at each boundary point of the pair of pants along a leaf of $\mathcal{F}$. During this isotopy,
the endpoint of $\alpha$ in $c_j$ passes a certain number $n$ of times through the distinguished point $p$ it contains and  $n$ comes with a sign, where a ``positive crossing"
corresponds to the specified orientation on $c_j$. 
Denoting by $m(c_j)$ the $\mathcal{F}$-transverse measure of $c_j$
and $\tau (c_j)$ the transverse measure along $c_j$ from $p$ to the endpoint of the perpendicular arc in the specified orientation, the {\it pre-twisting number} of the configuration
$(p,\alpha)$ is defined to be $t(p,\alpha)=\tau (c_j)+ n m(c_j)$.  As usual, the twisting number $t_j(\mathcal{F})$  for the boundary curve $c_j$
arises by comparing the sense of the pre-twisting with that of the orientation of a neighborhood of the curve.

If $C=c_j$ is now an interior pants curve, then with the above construction, there are two arcs $\alpha$ and $\beta$ with endpoints in $C$ and two pre-twisting numbers, one from each side of $C$. The twisting number of the distinguished point in $C$ is then defined as the difference of the two pre-twisting numbers on $C$ (again, after a choice made by orientations of which pre-twisting number is subtracted from the other).
 
We now show that the map $\mathcal{J}$ in the statement of Theorem \ref{thm:noDT} is a bijection.
Let $\mathcal{F}$ and $\mathcal{G}$ be two elements in $\mathcal{F}$ that have the same image under $\mathcal{J}$. If the intersection numbers of $\mathcal{F}$ and $\mathcal{G}$ with one of the curves in $\mathcal{P}$ are different, then it follows from the definition of the equivalence relation between measured foliations that the elements $\mathcal{F}$ and $\mathcal{G}$ are distinct.

Assume now that the intersection number of any curve in $\mathcal{P}$ with $\mathcal{F}$ is equal to its intersection number with $\mathcal{G}$. If $\mathcal{F}\not=\mathcal{G}$, then the twist parameters of  $\mathcal{F}$ and  $\mathcal{G}$ with respect to some curve in $\mathcal{P}$ are not the same. 

We treat the following two cases separately:

Case 1.--- Suppose first that this curve is a 1-sided curve $b_j\in \mathcal{P}_1$. By assumption, $i(b_j,\mathcal{F})=i(b_j,\mathcal{G})$, so the two parameters $n_j(\mathcal{F})$ and $n_j(\mathcal{G})$ are both nonpositive. This means that there is a foliation of a maximal M\"obius band in $\mathcal{F}$ and in $\mathcal{G}$ whose core curve is $b_j$ and all whose leaves are  double covers of this core curve, and such that the transverse measure of a complete transversal is not the same for $\mathcal{F}$ and for $\mathcal{G}$. (One of these two transverse measures might be zero.) We take such a maximal foliated M\"obius band in $\mathcal{F}$ and in $\mathcal{G}$ and we denote it by $M_1$ and $M_2$ respectively. (To spare notation, we shall denote by the same letters, $M_1$ and $M_2$, the measured foliations supported on the M\"obius bands $M_1$ and $M_2$.) Now the foliation induced by $\mathcal{F}$ (respectively $\mathcal{G}$) on $P$ is the union of $M_1$ (respectively $M_2$) with a collection of leaves that are in the complement of $M_1$ (respectively $M_2$) . Let $\gamma$ be the closed curve represented in Figure \ref{gamma}. 
\begin{figure}[hbt] 
 \begin{center}
 \includegraphics[width=12cm]{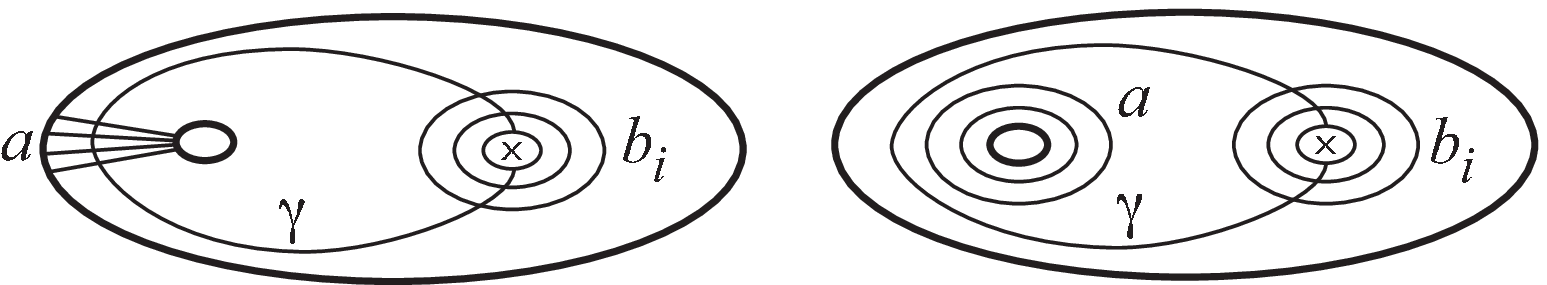}
 \caption{}
 \label{gamma} 
\end{center}
\end{figure}
In the pair of pants $P$, $\gamma$ restricts to an arc joining $b_j$ to itself. In one of the cases represented in the figure, $\mathcal{F}$ or $\mathcal{G}$ (or, equivalently, both of them) has zero intersection number with the two other bondary components of $S$. In this case, the intersection of $\gamma$ with $\mathcal{F}$ (respectively $\mathcal{G}$) is equal to $i(\gamma,M_1)+m_1$ (respectively $i(\gamma,M_2)+m_2$) where $m_1$ and $m_2$ are the intersection numbers of $\mathcal{F}$ and $\mathcal{G}$ respectively with the foliation induced on the pair of pants $P$ by   $\mathcal{F}$ and $\mathcal{G}$. Note that all the leaves of the induced foliation are arcs that go from the two boundary components of $P$ that are not $b_j$.  Let $a$ be either of the two boundary components of $P$ that is distinct from $b_j$. Thus, $m_1=i(\mathcal{F},a)$ and $m_2=i(\mathcal{G},a)$. However, by assumption $m_1=n_1$, so $i(\gamma,\mathcal{F})\not=  i(\gamma,\mathcal{G})$. This shows that $\mathcal{F}\not=\mathcal{G}$. In the other case represented in Figure \ref{gamma}, both $\mathcal{F}$ and $\mathcal{G}$ have zero intersection number with the two boundary components of $P$ that are distinct from $b_j$. In this case, we have $i(\gamma,\mathcal{F})= i(\gamma,M_1)$ and  $i(\gamma,\mathcal{G})= i(\gamma,M_2)$ which also gives 
 $i(\gamma,\mathcal{F})\not=i(\gamma,\mathcal{G})$, so $\mathcal{F}\not=\mathcal{G}$.

Case 2.--- Assume now that the curve in $\mathcal{P}$ on which the twist parameters are not the same for $\mathcal{F}$ and  $\mathcal{G}$ is a 2-sided curve $c_j$.
If $i(\mathcal{F},c_j)=i(\mathcal{F},c_j)=0$ (recall that we are under the assumption that the two intersection numbers are the same), then there are two maximal foliated annuli parallel to $c_j$ whose total $\mathcal{F}$-width and $\mathcal{G}$-width are different. This shows that $\mathcal{F}$ and $\mathcal{G}$ are not equivalent.

The case where $i(\mathcal{F},c_j)=i(\mathcal{F},c_j) \not=0$ and the twisting parameters are different is treated in the same way as the classical (orientable) case, see \cite{FLP}, Expos\'e 6, \S 2. On each side of the curve there is a pair of pants (which could be the same for the two sides), and in the union of these pair of pants (respectively in this pair of pants) it is possible to find a simple closed curve whose transverse measures with respect to $\mathcal{F}$ and  $\mathcal{G}$ are different. Thus, the foliations are not equivalent, completing the proof of the theorem.

To close the section, we comment on the case where the parameters are integers, that is, where the measured foliations are obtained by ``enlargement" of simple closed curves and proper arcs. This will lead us to a parametrization \`a la Dehn of the set of simple closed curves and proper arcs in $S$.

We first consider the case of an orientable surface $F$. 

There is an \emph{enlargement} operation (see \cite{FLP} Expos\'e 5 \S 3) which embeds the set $\mathcal{S}$ of homotopy classes of essential simple closed curves in $F$ into  $\mathcal{MF}(F)$. Likewise, we need to define an embedding of the set of arcs.
To this end, let $a$ be an essential arc joining two boundary components $\partial_1$ and $\partial_2$ of $F$ and let $r$ be a positive number. Let $N(a)$ be a regular neighborhood of $a$ in $F$, so $N(a)$ has the natural structure of a quadrilateral having two opposite sides $s_1$ and $s_2$ contained in $\partial_1$ and $\partial_2$ respectively, the other opposite sides being homotopic rel boundary to the arc $a$. We equip the quadrilateral $N(a)$ by a foliation with leaves homotopic rel boundary to the curve $a$, and we equip this foliation with a transverse measure such that the total measure of each of the two arcs $s_1$ and $s_2$ is equal to $r$. We then collapse  the surface with boundary $F'=F\setminus N(a)$ onto a one-dimensional spine in such a way that the image of each piece of $N(a)$ which is in the boundary of $F'$ injects into the spine. We get a measured foliation all of whose nonsingular leaves are arcs in the same homotopy class. The new surface is equipped with a canonical homotopy class of maps onto $F$ which sends the homotopy class of the nonsingular leaves to the homotopy class of the arc $a$. Taking the image of this foliation by one such map gives a well-defined equivalence class of measured foliations on $F$ with all non-singular leaves in the homotopy class $a$. This defines a map from the set of equivalence classes of weighted arcs on $F$ into the set of equivalence classes of measured foliations on $F$. 

This map is injective. To see this, we use the well-known fact \cite{FLP} that in the case of closed surfaces, the enlargement operation from the set of weighted simple closed curves to measured foliation space is injective.  It suffices to double the surface $F$ along its boundary components, obtaining a surface without boundary $F^d$. The operation of enlargement of arcs in $F$ gives, in the double, an operation of  enlargement of curves (the double of the arcs). The fact that the enlargement of curves operation gives an injection from the space $\mathbb{R}_{>0}\times \mathcal{S}(F^d)$ into the space $\mathcal{MF}(F^d)$ implies that the enlargement of arcs in $F$ gives an injection from the space $\mathbb{R}_{>0}\times \mathcal{C}(F)$ into the space $\mathcal{MF}(F)$. It is also clear that the image of the homotopy class of an arc in measured foliation space cannot coincide with the image of a homotopy class of a curve since the measured foliation associated to an arc has leaves which are transverse to the boundary, unlike the case of a measured foliation associated to a curve.

 In the case of a non-orientable surface $N$, the theory works in the same way except that the injectivity fails only at the level of the non-primitive 2-sided closed curves, whose enlargement coincides with the enlargement of the 1-sided curve which they double-cover. This is the principal reason for which we introduce the set $\Pi \mathcal{S}(S)$ versus the set $\mathcal{S}(S)$.
Thus, we have a natural embedding
\begin{equation} \label{eq:embedding-arcs}
\Pi{\mathcal S}(S)\hookrightarrow \mathcal{MF}(S).
\end{equation}

The fact that measured foliations can be put in normal position with respect any pants decomposition implies (using the enlargement operation) a similar statement for elements of $\mathcal{D}$ (curves and arcs), where in each pair of pants, such an element induces a system of arcs (perhaps empty) joining a boundary component to itself or two different boundary components, together with a system of curves parallel to boundary components.  Theorem \ref{thm:noDT} gives the following for curves and arcs:

\begin{theorem}[Generalized Dehn Theorem for curves and arcs]\label{thm:noDTC}

Fix a pants decomposition ${\mathcal P}={\mathcal P}_1\cup {\mathcal P}_2$ of $S=F_{g,r}$ or $S=N_{g,r}$, where ${\mathcal P}_1=\{b_1,\ldots ,b_M\}$ and
${\mathcal P}_2=\{ c_1,\ldots ,c_N\}$.  
Given an element $d$ in $\mathcal{D}(S)$ (that is, an isotopy class of a not necessarily connected 1-dimensional manifold of $S$ whose components are essential and non-boundary parallel) for each $b_i$, $j=1,\ldots, M$, define an integer parameter $n_j(d)$ as follows:
\begin{itemize}
\item if the intersection number of $d$ with $b_j$ is positive, then $n_j(d)\geq 0$ agrees with this intersection number;
\item if $d$ contains a closed leaf isotopic to $b_j$ (and in that case, the intersection number of $d$ with $b_j$ is zero), then $n_j(d)\leq 0$ is  given by
$-2k$ if $d$ consists of $k$ disjoint copies of the 2-sided double cover of $b_j$ and by 
$-2k-1$ if $d$ consists of $k$ disjoint copies of the double cover of $b_j$ plus a copy of $b_j$ itself.
\end{itemize}
 For an element $c_j$ of $\mathcal{P}_2$, we let $m_j(d)$ and $t_j(d)$ denote the intersection and twisting numbers of $d$ with $c_j$, for $j=1,\ldots ,N$, computed as in Theorem \ref{thm:noDT} for measured foliations.
Then the mapping
$$
\aligned
{\mathcal{K}: \mathcal {D}}(S)&\to {\mathbb Z}^M\times (({\mathbb N}_{>0}\times{\mathbb Z})\cup (\{0\} \times {\mathbb Z}))^N\cr
\mathcal F&\mapsto (n_1(\mathcal D),\ldots, n_M(\mathcal D),m_1(\mathcal D),t_1(\mathcal D),\ldots ,m_N(\mathcal D),t_N(\mathcal D))\cr
\endaligned
$$
is a bijection.  
\end{theorem}

The proof of the theorem follows from that of Theorem \ref{thm:noDT}, using the injectivity of enlargement map that we already described. 

\section{The Klein Bottle Minus a Disk}

Let $K_1=N_{2,1}$ denote the Klein bottle minus an open disk and let $M_1=N_{1,2}$ denote the M\"obius band minus an open disk.

\begin{lemma}\label{lem:2pd}
Up to isotopy, there are exactly two essential and non-peripheral (i.e., non-boundary parallel) simple closed curves in $M_1$.  Each such curve is 1-sided, and these are exactly (up to isotopy) the two possible pants decompositions of this surface.
\end{lemma}
\begin{proof} 
The M\"obius band $M$ itself contains a unique (up to isotopy) essential simple closed curve which is not boundary parallel, namely, its core $c$,
so any pants curve in  $M_1$ is isotopic to $c$ in $M$. Such a curve in $M_1$ can run either below or above
the boundary of the removed disk (referring to Figure \ref{M1}).  There is a unique such isotopy class in $M_1$ in each case, and these two curves provide the two pants decompositions
of $M_1$.  Notice that these curves meet transversely in a single point.
\end{proof}

\begin{figure}[hbt]%
 \begin{center}
 \includegraphics[width=6cm]{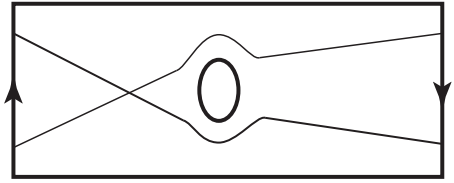}
 \caption{The two essential simple closed curves which are not boundary-parallel in $M_1$, the M\"obius band minus a disk.}
 \label{M1}%
\end{center}
\end{figure}

\begin{lemma}\label{lem:cde} For each embedded 1-sided curve $d$ in $K_1$, there are exactly two
embedded 1-sided curves $c,e$ disjoint from $d$, and furthermore, $c$
and $e$ have geometric intersection one.
\end{lemma}

\begin{proof}
Cut $K_1$ on $d$ to produce a M\"obius band minus a disk $M_1$, which contains exactly two embedded 1-sided curves by the previous lemma,
namely, the core of the M\"obius band running either above or below the missing disk.  These two curves meet transversely
in a single point and upon regluling along $d$ give the required curves $c,e$ in $K_1$.  
\end{proof}

\begin{lemma}\label{lem:onlyone}
$K_1$ contains exactly one primitive non-peripheral 2-sided curve we shall denote $c_2$.
\end{lemma}

\begin{proof}
Remove a regular neighborhood of two disjoint 1-sided curves in $K_1$ to produce a pair of pants $P$, 
one boundary component of which arises from the boundary of $K_1$ and the other two of which are non-primitive in $K_1$.
By the Dehn-Thurston Theorem \ref{thm:noDT}, any primitive non-peripheral 2-sided curve $c$ in $K_1$
restricts to standard position in $P$ comprised of $p$ arcs connecting different boundary components
and $q$ arcs connecting one boundary component to itself (and disjoint from the third boundary component of $P$ arising
from the boundary of $K_1$).  Since $c$ is 2-sided, $p$ must be even, and $p=2$, $q=0$
gives the unique 2-sided curve meeting exactly once each of the two 1-sided curves in $K_1$ whose neighborhood we remove.
Furthermore, if $q=0$ and $p>2$, then the resulting one-manifold in $K_1$ is not connected.

Thus, we may suppose that $q\neq 0$.  2-sidedness implies that $p+q$ is even, so $q$ is even since $p$ is.
We may thus pair up consecutive parallel arcs from $c$ and build a new arc family in $P$ where there is one new
arc between each such pair of original consecutive parallel arcs.  Upon regluing to get $K_1$, the new arc family gives rise to an embedded
closed 1-submanifold of $K_1$ that is twice covered by $c$.  See Figure \ref{fig:cover} for an example. In particular, the resulting
1-submanifold is connected since $c$ is, and this contradicts primitivity.

\begin{figure}[hbt]%
 \begin{center}
 \includegraphics[width=6cm]{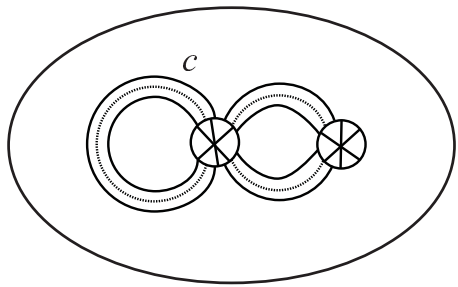}
 \caption{Construction of a 1-sided curve from a non primitive 2-sided curve $c$ in $K_1$ when both $p$ and $q$ are even.}
 \label{fig:cover}%
\end{center}
\end{figure}

\end{proof}

The set of isotopy classes of 1-sided curves in $K_1$ is the set of vertices of a simplicial graph called the \emph{graph of 1-sided curves}, where an edge joins two vertices if and only if the vertices can be represented by disjoint curves.

\begin{corollary}\label{cor:K1cpx}
The graph of 1-sided curves in $K_1$ is isomorphic as a simplicial complex to the real line with vertices given by the integers, and the Dehn twist on the unique primitive non-peripheral 2-sided curve $c_2$ (see Lemma \ref{lem:onlyone}) acts by translation by one on this graph.
\end{corollary}

\begin{proof}  The first part follows immediately from Lemma \ref{lem:cde}, and the last part is illustrated
in Figure \ref{fig:trans}.
\end{proof}

\begin{figure}[hbt]%
 \begin{center}
 \includegraphics[width=8cm]{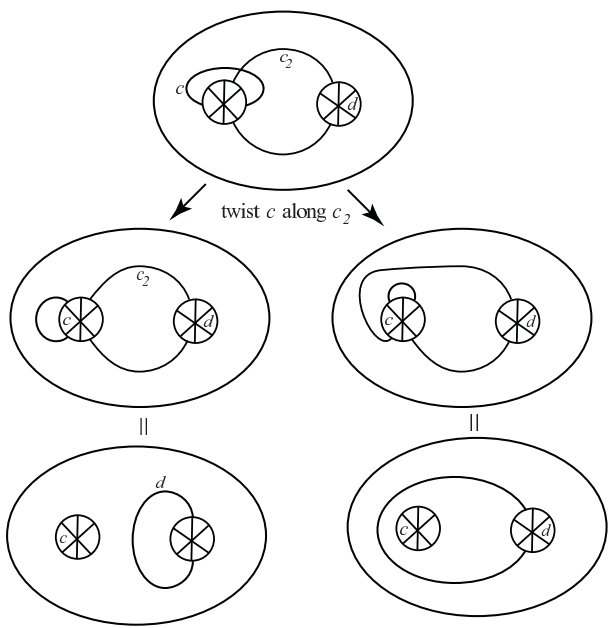}
 \caption{The effect on $c$ of the right and left Dehn twists along $c_2$ in $K_1$.  One twist produces $d$ whereas the other
 arises from move IV replacing $d$ by the illustrated curve.}
 \label{fig:trans}%
\end{center}
\end{figure}

\begin{proposition}\label{prop:K1HT} 
Finite compositions of the two moves III and IV depicted in Figure \ref{fig:elementary} act transitively on pants decompositions of $K_1$.
\end{proposition}

\begin{proof}
A pants decomposition comprised of two 1-sided curves corresponds precisely to a 1-simplex in the graph of 1-sided curves
in $K_1$ and is described in Corollary \ref{cor:K1cpx}.  As already depicted in Figure \ref{fig:trans}, translation by one in this 1-sided curve complex
is achieved by two consecutive moves of type III, so finite compositions of  moves of type III indeed act transitively on
pants decompositions in $K_1$ comprised of 1-sided curves.  According to Lemma \ref{lem:onlyone}, the move IV applied to
any such pants decomposition produces the curve $c_2$, namely, the unique pants decomposition containing a 2-sided
curve.
\end{proof}

One of our basic inductive tools is provided by:

\begin{lemma}\label{lem:self}
Given three distinct 1-sided curves $c,d,e$ in $K_1$, where $c$ and $d$ are disjoint and $e$ meets $c\cup d$ minimally,
there is a closed sub-arc of $e$ whose interior is disjoint from $c\cup d$ and whose endpoints
lie on a common component $c$ or $d$.
\end{lemma}

\begin{proof}
As in the proof of Lemma \ref{lem:onlyone}, remove a regular neighborhood of $c\cup d$ to produce a pair of pants and consider
the restriction of $e$ to this pair of pants in its standard position comprised of $p\geq 0$ arcs connecting $c$ to $d$ and $q\geq 0$ arcs
running from one component $c$ or $d$ back to itself.  We must show $q\neq 0$, and in the contrary case, $p$ must be even since
$e$ is a closed curve while $p$ must be odd in order that $e$ is 1-sided.  This contradiction establishes $q\neq 0$ as required.
\end{proof}

\section{Generalized Hatcher-Thurston Theorem}

Throughout this section, $S$ will denote a compact possibly non-orientable surface of negative Euler characteristic.
Consider the four combinatorial moves on pants decompositions illustrated in Figure \ref{fig:elementary}.  In each case, each of
the boundary curves might bound a M\"obius band (e.g., graphically, may perhaps be drawn containing an asterisk)  and must retain this attribute before and after the move.  The moves
are to be applied on any subsurface of $S$ bounded by pants curves.  
%Note the special
%move V which is applied only on the closed surface $S$ given by the connected sum of three projective planes, a special case whose discussion is simply
%postponed (to the end of the proof of Theorem \ref{thm:ght}).

Recall that ${\mathcal P}_1$ denotes the set of 1-sided curves in the pants decomposition ${\mathcal P}$ of $S$.  The pants decomposition ${\mathcal P}$ is said to be {\it orientable} if $S-\cup{\mathcal P}_1$
is an orientable surface.  
%In this case, there are two orientations on the set of curves in ${\mathcal P}_1$ consistent with the orientations of $S-\cup{\mathcal P}$.
\begin{lemma} \label{lem:opd} For any pants decomposition of a possibly non-orientable surface, there is a finite sequence of moves I--IV
resulting in
an orientable pants decomposition.
\end{lemma}

\begin{proof}  Given a pants decomposition ${\mathcal P}$ of $S$, let us build the usual corresponding ``dual'' graph  (i.e., 1-dimensional CW complex) $G=G({\mathcal P})$ with one trivalent vertex for each complementary pair of pants and one edge for each pants curve connecting vertices corresponding to pairs of pants on the two sides of a 2-sided pants curve.  In our case, a boundary component gives rise to a univalent vertex, and a 1-sided curve gives rise to a univalent vertex marked with an X; see Figure \ref{fig:ptog} for an example.

\begin{figure}[hbt]%
 \begin{center}
 \includegraphics[width=10cm]{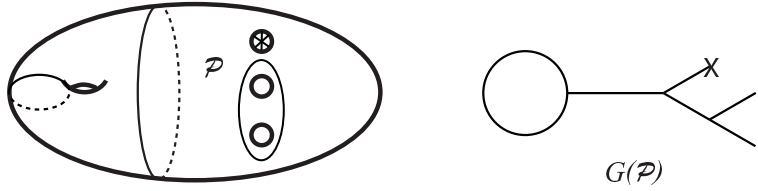}
 \caption{Example of a graph $G({\mathcal P})$ associated to a pants decomposition ${\mathcal P}$ of $N_{3,2}$.}
 \label{fig:ptog}%
\end{center}
\end{figure}

Choose a maximal tree $T$ in $G$ and choose an orientation
on the associated open subsurface  determined as the complement of the edges (i.e., pants curves) $E$ of $G$ that are not in $T$.
Each $e\in E$ determines a vertex in $G$ which is represented by a curve in $S$ which may be 1- or 2-sided, and in the 
latter case, we may add $e$ to $T$ and preserve
orientability of the corresponding surface.

Thus, the connected subsurface $\Sigma\subseteq S$ complementary to a neighborhood of the pants curves in $\{ e\in E:\Gamma_e~{\rm is~1-sided}\}\cup{\mathcal P}_1$
is orientable.  The boundary components of $\Sigma$ are either cross caps (duals marked with an X), boundary components of $S$ itself,  or occur in pairs $\{ c,d\}$ which are identified so as to reverse the orientation on $\Sigma$.  By the Hatcher-Thurston Theorem in the orientable surface $\Sigma$, we may arrange using moves I and II that there is a complementary pair of pants bounded by $c, d\subset\partial\Sigma$ and another pants curve $e$, which thus separates off a subsurface $K_1=N_{2,1}$.  Applying the move IV in this surface $K_1$ replaces one 2-sided curve in ${\mathcal P}$ by two 1-sided curves in a new pants decomposition ${\mathcal P}'$ with ${\mathcal P}_1'={\mathcal P}_1+2$.
There is an upper bound of $g$ to the number of cross caps on a surface $N_{g,r}$, and this procedure thus terminates with an orientable pants decomposition.
\end{proof}

\begin{lemma}\label{lem:1or2}
For any pants decomposition of a possibly non-orientable surface, there is a finite sequence of moves I--IV producing 
an orientable one ${\mathcal P}$ with
$\#{\mathcal P}_1\leq 2$.
\end{lemma}

\begin{proof} By Lemma \ref{lem:opd}, we may assume that the surface $S-\cup {\mathcal P}_1$ is oriented, and certain of its boundary components
(all of them if $S$ is closed) correspond to the 1-sided curves.  By the classical Hatcher-Thurston Theorem,
the usual moves I-II act transitively
to bring together any triples of these 1-sided curve boundary components into a subsurface
homeomorphic to $N_{3,1}$ whose boundary is a pants curve containing furthermore
another pants curves separating two of the 1-sided curve boundary components.   The sequence of
moves supported on $N_{3,1}$ illustrated in Figure \ref{fig:tildeF} shows how to reduce the number of 1-sided curves by two,
as required.  The unique exception is the surface $N_{3,0}$, where the sequence of moves illustrated in
Figure \ref{fig:N30} accomplishes the required task.
\end{proof}

Notice that if ${\mathcal P,\mathcal Q}$ are pants decompositions of $S$ with $\#{\mathcal P}_1,\#{\mathcal Q}_1\leq 2$,
then $\#{\mathcal P}_1=\#{\mathcal Q}_1$ is a topological property of the surface $S$, i.e., the parity of $\#{\mathcal P}_1$
is an invariant for non-orientable surfaces.

\begin{figure}[hbt]%
 \begin{center}
 \includegraphics[width=10cm]{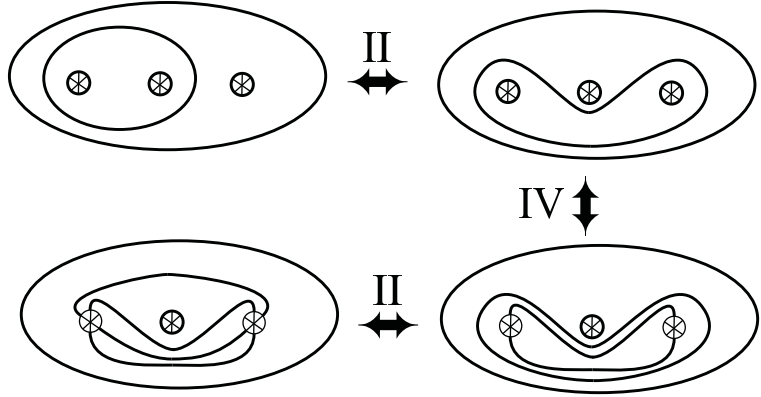}
 \caption{A sequence of moves in $N_{3,1}$ preserving orientability and replacing three 1-sided pants curves and one 2-sided one
 by one 1-sided curve and two 2-sided curves.}
 \label{fig:tildeF}%
\end{center}
\end{figure}

\begin{figure}[hbt]%
 \begin{center}
 \includegraphics[width=10cm]{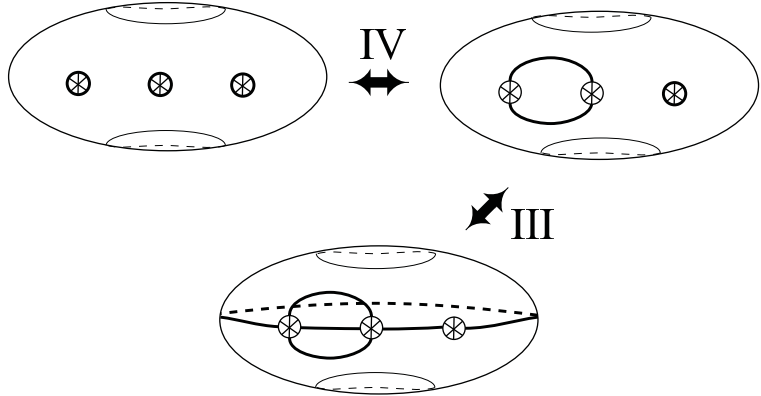}
 \caption{A sequence of moves in $N_{3,0}$ preserving orientability and replacing three 1-sided pants curves 
 by one 1-sided curve and one 2-sided curve.}
 \label{fig:N30}%
\end{center}
\end{figure}

Given a pants decomposition ${\mathcal P}$ and a simple closed curve $c$, we set $i(c,{\mathcal P}')=\sum_{c'\in {\mathcal P}'} i(c,c')$.
Given two pants decomposition ${\mathcal P}, {\mathcal P}'$ of $S$, we define a relative {\it pre-complexity}
$$\gamma_1({\mathcal P},{\mathcal P}') = ~{\rm inf}_{c\in {\mathcal P}_1} i(c,{\mathcal P}'),$$
where $i(\cdot,\cdot)$ denotes the geometric intersection number and $\gamma_1({\mathcal P},{\mathcal P}')$ is taken to be
infinite if ${\mathcal P}_1$ is empty.  The symmetrization 
$$\gamma({\mathcal P},{\mathcal P}')=~{\rm inf}\{\gamma_1({\mathcal P},{\mathcal P}'),
\gamma_1({\mathcal P}',{\mathcal P})\}$$
is the {\it complexity} upon which we shall induct.

Taking Dehn-Thurston coordinates from Theorem \ref{thm:noDT} with respect to
${\mathcal P}'$, it follows that if $\gamma({\mathcal P},{\mathcal P}') =0$, then ${\mathcal P}_1\cap{\mathcal P}'_1\neq\emptyset$.
Furthermore with the previous lemma in mind for $\#{\mathcal P},\#{\mathcal P}'\leq 2$ , if ${\mathcal P}_1=\{ c\}$ and ${\mathcal P}'=\{ c'\}$, then 
$\gamma({\mathcal P},{\mathcal P}') =
~{\rm inf}\{ i(c,{\mathcal P}'),i(c',{\mathcal P})\}$
whereas if ${\mathcal P}_1=\{ c,d\}$ and ${\mathcal P}'=\{ c',d'\}$, then 
$$\gamma({\mathcal P},{\mathcal P}') =~{\rm inf}\biggl\{ {\rm inf}\{i(c',{\mathcal P}),i(d',{\mathcal P}')\}, {\rm inf}\{i(c,{\mathcal P}'),i(d,{\mathcal P}')\}\biggr\}.$$

\begin{theorem}[Generalized Hatcher-Thurston Theorem]  \label{thm:ght} Finite compositions of moves I-IV act transitively on pants decompositions
of surfaces $S=F_{g,r}$ or $S=N_{g,r}$.
\end{theorem}

\begin{proof}  We proceed by induction on the Euler characteristic $\chi(S)$ of $S$, and the basis step $\chi(S)=-1$ is covered
by $S=F_{0,3}$ or $S=N_{g,r}$ for $g+r=3$.   There is the unique pants decomposition (by boundary curves) in the oriented case.
Lemma \ref{lem:cde} treats $S=N_{1,2}$, and Proposition \ref{prop:K1HT} covers $S=N_{2,1}$.  We shall postpone the
special discussion of $S=N_{3,0}$, namely, the connected sum of three projective planes and no boundary, until the end. 
(Since the inductive step in all other cases devolves to a surface with boundary, it is legitimate to separately treat this case in this way.)
Within the induction over $\chi(S)\leq -2$, we shall further induct upon the complexity.  Furthermore, if $S$ is orientable, then the classical Hatcher-Thurston theorem
applies.

Given two pants decompositions ${\mathcal P},{\mathcal P}'$ of $S$, we apply Lemma \ref{lem:1or2} to without loss of generality assume that ${\mathcal P},{\mathcal P}'$ are oriented with  $\#{\mathcal P}_1= \#{\mathcal P}'_1=1$ or $\#{\mathcal P}_1= \#{\mathcal P}'_1=2$.
Now induct over the complexity $\gamma=\gamma({\mathcal P},{\mathcal P}')$.
As noted before, if $\gamma=0$, then there is a 1-sided curve $c$ common to ${\mathcal P}$ and ${\mathcal P}'$ by the Generalized Dehn-Thurston Theorem, and we may cut on $c$ to decrease the Euler characteristic and apply the inductive hypothesis.

Let us suppose now that $\gamma>0$ and consider separately the two cases $\# {\mathcal P}_1= \#{\mathcal P}'_1=1$ or $2$.
A convenient notation here and throughout is to write simply ${\mathcal P}\rightarrow {\mathcal Q}$ if the pants decompositions ${\mathcal P}$ and ${\mathcal Q}$ are related
by moves I-IV.  This binary relation is in fact an equivalence relation by definition.

In the former case, again ${\mathcal P}_1=\{ c\}$ and ${\mathcal P}_1'=\{ c'\}$, where $c\cap c'\neq 0$ since ${\mathcal P}'$ is orientable.  
First, consider the case that $i(c,c')=1$, in which case a regular neighborhood $N=N(c\cup c')$ of $c\cup c'$ in $S$ is homeomorphic to $N_{1,2}$ with $c,c'$ providing the
two pants decompositions discussed before.  Extend $\{\partial N,c\}$ to a pants decomposition ${\mathcal P}''$ of $S$.  Since $c\in {\mathcal P}\cap{\mathcal P}''$, we may cut on $c$ and conclude that ${\mathcal P}\rightarrow{\mathcal P}''$ by the inductive hypothesis on Euler characteristic.  Likewise, we have
$c'\in ( ({\mathcal P}''-\{ c\})\cup\{ c'\})\cap {\mathcal P}'$, so also ${\mathcal P}'\rightarrow {\mathcal P}''$, and it follows by transitivity that
${\mathcal P}\rightarrow{\mathcal P}'$.  Thus, we may henceforth suppose that $i(c,c')>1$.

\begin{figure}[hbt]%
 \begin{center}
 \includegraphics[width=6cm]{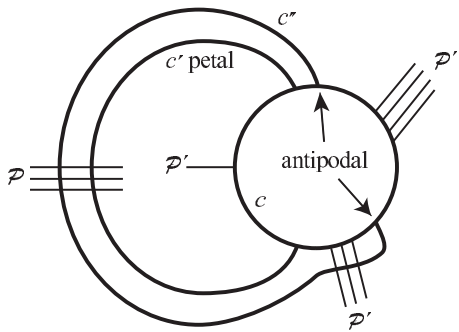}
 \caption{Construction of the curve $c''$ from a petal of $c'$.}
 \label{fig:firstcase}%
\end{center}
\end{figure}

For the purposes of the proof, a component of $c'-c$ is called a ``petal''; equivalently, a petal is one of the subarcs of $c'$ arising between
two points of $c$ consecutive along $c'$.
Choose such a petal and let $c''$ be the simple closed curve formed by connecting the petal endpoints  with a small subarc of  $c$ 
 as illustrated in Figure \ref{fig:firstcase}.  Thus, 
 $$  i(c'',{\mathcal P}')\leq i(c,{\mathcal P}').$$
 Equality holds here only if the endpoints of the petal are also consecutive along $c$.  In this case, there is a simple cycle $\alpha$ in $S$ comprised of the petal
 from $c'$ together with this subarc of $c$ (which is itself a petal for $c$).  Extend $\{\alpha ,c\}$ to a pants decomposition ${\mathcal P}''$
 and $\{\alpha ,c'\}$ to ${\mathcal Q}''$. As before, induction on Euler characteristic shows that 
 $c\in{\mathcal P}\cap{\mathcal P}''$ implies ${\mathcal P}''\rightarrow {\mathcal P}$, $c'\in{\mathcal P}'\cap{\mathcal Q}''$ implies ${\mathcal P}'\rightarrow {\mathcal Q}''$,
 and $\alpha\in{\mathcal P}''\cap{\mathcal Q}''$ implies ${\mathcal P}''\rightarrow {\mathcal Q}''$, whence  ${\mathcal P}\rightarrow{\mathcal P}'$ as required.
  Thus, we may henceforth assume strict inequality $i(c'',{\mathcal P}')<i(c,{\mathcal P}')$.
  
Now build new pants decomposition ${\mathcal P}''$ and ${\mathcal Q}''$ (using this same notation from above for variable pants decompositions now), where ${\mathcal P}''$ contains
$\{c'',\partial N \}$  and ${\mathcal Q}''$ contains $\{ c,\partial N\}$ with $N=N(c\cup c'')$.  It follows by induction on Euler characteristic that ${\mathcal P}\rightarrow{\mathcal Q}''$,
and in fact also ${\mathcal P}\rightarrow {\mathcal P}''$ since ${\mathcal P}''$ and ${\mathcal Q}''$ differ by a move III in $N$ by construction.
Furthermore, we have
$$\gamma({\mathcal P}',{\mathcal P}'')\leq\gamma_1({\mathcal P}'',{\mathcal P}')\leq i(c'',{\mathcal P}')<i(c,{\mathcal P}')=\gamma_1({\mathcal P},{\mathcal P}'),$$
where the first inequality follows from the fact that $\gamma$ is also defined as an infimum, the second since $\gamma_1$ is defined as an infimum, the third (strict) inequality as already discussed, and the last equality since ${\mathcal P}_1=\{ c\}$ in the case we are considering.  Thus, $\gamma({\mathcal P}',{\mathcal P}'')\leq \gamma_1({\mathcal P},{\mathcal P}')$ with ${\mathcal P}\rightarrow{\mathcal P}''$, and interchanging the roles of ${\mathcal P}$ and ${\mathcal P}'$ provides the analogous
$\gamma({\mathcal P},{\mathcal Q}'')<\gamma_1({\mathcal P}',{\mathcal P})$ with ${\mathcal P}'\rightarrow{\mathcal Q}''$.

Since $\gamma({\mathcal P},{\mathcal P}')$ is defined as the infimum of 
$\gamma_1({\mathcal P},{\mathcal P}')$ and $\gamma_1({\mathcal P}',{\mathcal P})$, it follows that 
either $\gamma({\mathcal P}',{\mathcal P}'')<\gamma({\mathcal P},{\mathcal P}')$ with ${\mathcal P}\rightarrow{\mathcal P}''$
or $\gamma({\mathcal P},{\mathcal Q}'')<\gamma({\mathcal P},{\mathcal P}')$ with ${\mathcal P}'\rightarrow{\mathcal Q}''$, so
in any case,
${\mathcal P}\rightarrow{\mathcal P}'$ by induction on complexity as required in the first case.

\begin{figure}[hbt]%
 \begin{center}
 \includegraphics[width=11cm]{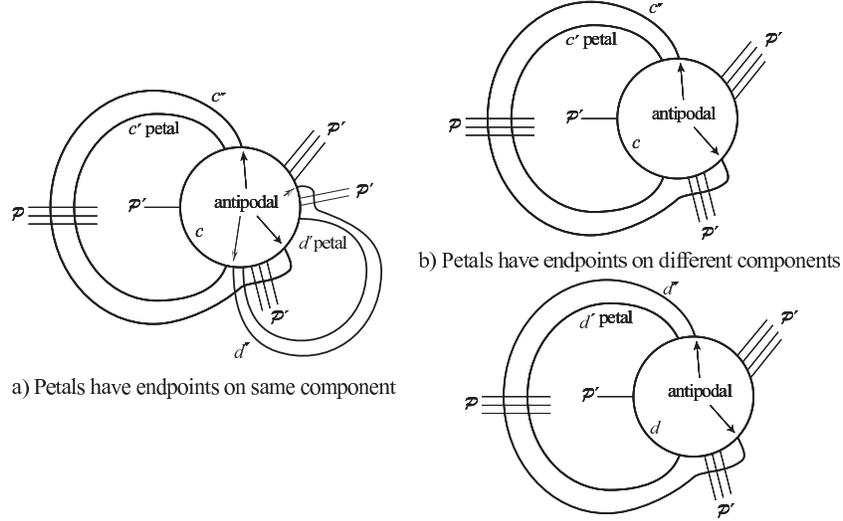}
 \caption{Construction of curves $c'',d''$ from petals of $c',d'$.}
 \label{fig:secondcase}%
\end{center}
\end{figure}

In the latter case, let ${\mathcal P}_1=\{ c,d\}$ and ${\mathcal P}'_1=\{ c',d'\}$ and define a petal of $c'$ or $d'$ to be component of $c'-(c\cup d)$ or $d'-(c\cup d)$, respectively,
whose endpoints lie on a common component of $c\cup d$.  
We claim that there must be a petal of $c'$ with both endpoints either on $c$ or on $d$, for otherwise, all consecutive subarcs of $c'$ with endpoints on $c\cup d$ must have one endpoint on $c$ and one on $d$, say for a total of $m$ such arcs; however, $m$ must be even since $c'$ is closed, and yet $m$ must be odd since $c'$ is 1-sided.  Thus, there indeed must be a petal of $c'$ with endpoints on the same pants curve $c$ or $d$ and likewise for $d'$.
Up to relabeling the pants curves, there are two possible scenarios: either there are petals of both $c'$ and $d'$ with endpoints on $c$, or $c'$ has a petal with endpoints on $c$ and $d'$ has a petal with endpoints on $d$.  Build new curves $c''$ and $d''$ as before as illustrated in Figure \ref{fig:secondcase} in the two scenarios, where perhaps $c''\cap d''\neq \emptyset$ in either case.  

Construct pants decompositions ${\mathcal P}'' \supset \{ d,\partial N(c\cup c'')\}$ and ${\mathcal Q}''\supset\{ d,\partial N(c\cup d'')\}$ in the former scenario
and ${\mathcal Q}''\supset\{ d,\partial N(d\cup d'')\}$ in the latter. 
Thus, ${\mathcal P}\rightarrow{\mathcal P}''\rightarrow{\mathcal Q}''$ since $d\in {\mathcal P}\cap {\mathcal P}''\cap{\mathcal Q}''$.  By construction, we have
\begin{equation}\label{eqn:first}
i(c'',{\mathcal P}')\leq i(c,{\mathcal P}')~{\rm and}~i(d'',{\mathcal P}')\leq i(d,{\mathcal P}'),
\end{equation}
so
$$\begin{aligned}
\gamma({\mathcal P}'',{\mathcal P}')&\leq\gamma_1({\mathcal P}'',{\mathcal P}')\leq i(c'',{\mathcal P}')\leq i(c,{\mathcal P}'),\\
\gamma({\mathcal Q}'',{\mathcal P}')&\leq\gamma_1({\mathcal Q}'',{\mathcal P}')\leq i(d'',{\mathcal P}')\leq i(d,{\mathcal P}'),\\
\end{aligned}$$
where the two leftmost inequalitites hold since $\gamma$ and $\gamma_1$ are defined as infima.
Thus, ${\rm inf}\{(\gamma({\mathcal P}'',{\mathcal P}'),\gamma({\mathcal Q}'',{\mathcal P}')\}\leq \gamma_1({\mathcal P},{\mathcal P}')$
with ${\mathcal P}\rightarrow {\mathcal P}''$ and ${\mathcal P}\rightarrow {\mathcal Q}''$.  Moreover by construction, we also have
\begin{equation}\label{eqn:second}
i(c'',{\mathcal P})\leq i(c',{\mathcal P})~{\rm and}~i(d'',{\mathcal P})\leq i(d,{\mathcal P}),
\end{equation}
so
$$\begin{aligned}
\gamma({\mathcal P}'',{\mathcal P})&\leq\gamma_1({\mathcal P}'',{\mathcal P})\leq i(c'',{\mathcal P})\leq i(c,{\mathcal P}),\\
\gamma({\mathcal Q}'',{\mathcal P})&\leq\gamma_1({\mathcal Q}'',{\mathcal P})\leq i(d'',{\mathcal P})\leq i(d,{\mathcal P}),\\
\end{aligned}$$
and so 
${\rm inf}~\{(\gamma({\mathcal P}'',{\mathcal P}),\gamma({\mathcal Q}'',{\mathcal P})\}\leq \gamma_1({\mathcal P}',{\mathcal P})$
with ${\mathcal P}\rightarrow{\mathcal P}''$ and  ${\mathcal P}\rightarrow{\mathcal Q}''$.
Again, if strict inequalities hold in  equations (\ref{eqn:first}) and (\ref{eqn:second}), then the desired result ${\mathcal P}\rightarrow{\mathcal P}'$ follows by induction on complexity,
and it remains only to analyze the extreme cases of equality in equations (\ref{eqn:first}) and (\ref{eqn:second}).

If the equality $i(c'',{\mathcal P}')=i(c,{\mathcal P}')$ holds, then again the endpoints of the petal of $c'$ must be consecutive along $c$,
and there is an essential simple cycle $\alpha$ disjoint from $c$ and $c'$.  Thus, ${\mathcal P}\rightarrow {\mathcal P}'$ as required by induction
on Euler characteristic as before with a similar argument for the equality $i(d'',{\mathcal P}')=i(d,{\mathcal P}')$.  If $c'$ intersects $\cup{\mathcal P}$ at a point not in the $c'$ petal, then $i(c'',{\mathcal P})<i(c',{\mathcal P})$.  In particular, $i(c'',{\mathcal P})=i(c',{\mathcal P})$ implies that $i(c',d)=0$.  Let ${\mathcal P}''$ be a pants decomposition containing $\{ c',d\}$ so that ${\mathcal P}\rightarrow {\mathcal P}''$ and ${\mathcal P}'\rightarrow {\mathcal P}''$ by induction on Euler characteristic, whence ${\mathcal P}\rightarrow{\mathcal P}'$ as required.

It remains only to consider the special surface $S=N_{3,0}$.  The surface $S$ is filled by triples of disjoint 1-sided curves, and we choose a particular one represented as three cross caps on the sphere, thus making a choice of pants decomposition ${\mathcal P}$ of $S$,
and label the cross caps $a,b,c$.  Now consider any other pants decomposition ${\mathcal P}'$ of $S$, which must have at least one (either one or three) 1-sided curves, and choose a curve $c'\in{\mathcal P}'_1$.  If there is a component of $c'-(a\cup b\cup c)$ which has its endpoints on the same component of $a\cup b \cup c$, then the previous argument based on petals applies to conclude that ${\mathcal P}\rightarrow {\mathcal P}'$.  In the contrary case, suppose there are $p$ complementary components
of $c'-(a\cup b\cup c)$ connecting $a$ and $b$, $q$ connecting $b$ and $c$, and $r$ connecting $c$ and $a$.  Each of $p+q$, $q+r$ and $r+q$ must be even, so one of $p,q,r$ is even if and only if all three are even.  If all three are even, then the curve described by $p,q,r$ is comprised of two parallel copies, which is absurd since $c'$ is connected.
Thus, each of $p,q,r$ is odd.  It follows that $c'$ contains the curve described by $p=q=r=1$, and again since $c'$ is connected, it coincides with this curve. Cutting $S$ on $c'$ produces a torus minus a disk, and finite iterations of the move I act transitively on pants decompositions of this surface
by the classical Hatcher-Thurston theorem.  It follows that finite compositions of moves I-IV  act transitively on 
pants decompositions of  $S=N_{3,0}$ as well.
\end{proof}

\section{Applications and Perspectives}\label{sec:apps}

In a companion paper \cite{PPY}, we give an explicit description in coordinates of the action of the mapping class group on arcs and curves in the spirit and sense of \cite{Penner-thesis} using the coordinates we provide in section~\ref{s:Dehn-Thurston}.  It seems clear that both systems of generalized Fenchel-Nielsen and Dehn-Thurston coordinates should be generally useful.  Other intriguing questions include what vestige  of the Weil-Petersson metric or its 
K\"ahler two-form might remain in the non-orientable case and the related question of how one might go about quantizing Teichm\"uller theory in this case.

\appendix\relax

\section{Thurston's boundary for Teichm\"uller space}\label{app:a}

Let $S$ denote a possibly non-orientable and possibly bordered compact surface $S$.
The closure $\overline{{\mathcal T}}(S)$
 of the image of the map
${\mathcal T}(S)\hookrightarrow P\mathbb{R}_+^{\Pi{\mathcal S}(S)}$ (recalled in \S \ref{s:background}) is compact and the boundary $\overline{{\mathcal T}}(S)\setminus {\mathcal T}(S)$ is the image of the map $
{\mathcal {PF}}(S)\hookrightarrow P\mathbb{R}_+^{\Pi{\mathcal S}(S)}$  (also recalled in \S \ref{s:background}). One gets local charts near a point of the boundary by fixing a pair of pants decomposition $\mathcal{P}$ and associating to each hyperbolic structure $\rho$ on $S$ a canonical foliation $\mathcal{F}(\rho)$ obtained by taking pieces of equidistant curves to some geodesic curves in the pairs of pants as in \cite{FLP}, Expos\'e 8. This gives a homeomorphism between $\mathcal{MF}(S)$ and a subset of $\mathcal{MF}(S)$ consisting of measured foliations transverse to each curve in $\mathcal{P}$. There is a ``fundamental inequality" which compares the lengths of any element of $\Pi{\mathcal S}(S)$ to the intersection number of this curve with the foliation $\mathcal{F}(\rho)$, and using this inequality, one defines charts near the boundary points of the union $\mathcal{T}(S)\cup \mathcal{PF}(S)$. The proof imitates that given in \cite{FLP} (Expos\'e 8) for the case of an orientable surface. 

\section{Punctured surfaces}\label{app:zero}

In contrast to the bordered surfaces considered in the main text, we discuss surfaces with punctures in this appendix, that is,
we study a surface $S$ that is either the orientable surface $F_g^s$ or non-orientable surface $N_g^s$ of genus $g\geq 0$ with $s\geq 0$ punctures or distinguished points and no boundary.
The definition of a pants decomposition ${\mathcal P}$ of $S$ is exactly as before, namely, complementary regions to 
${\mathcal P}$ in $S$ are homeomorphic to the interior of a pair of pants including now the possibility that a complementary region might 
be a once- or twice-punctured disk or a thrice-punctured sphere, where we can imagine each boundary component in a bordered surface collapsed to a distinct puncture.  

The elementary moves of Figure \ref{fig:elementary} are likewise interpreted: any depicted boundary component might instead represent a puncture, a boundary component or a cross-cap both before and after the move.  The statement and proof of the generalized Hatcher-Thurston Theorem for punctured surfaces follows verbatim the earlier discussion.

The statement and proof of the generalized Fenchel-Nielsen Theorem for punctured surfaces  is also  literally unchanged for a possibly non-orientable and possibly punctured surface $S$.  In effect,  any boundary component is replaced by a puncture in $S$, and both parameters for a boundary component of a bordered surface are simply dropped.  (Indeed, it is natural to imagine the length parameters as vanishing with twists undefined.) There are still moduli that determine the locations of the punctures within the surface, but the parameters for the pants curves in $\partial S$ themselves are dropped.  

The generalized Dehn-Thurston Theorem also pertains essentially without change, again dropping intersection and twisting numbers
for pants curves representing punctures.  Yet another  variant we shall not further discuss here drops only the twisting numbers but keeps the intersection numbers of the vestigial boundary components; this version importantly includes arc complexes of punctured surfaces.

However, there is more to explain in the context of Thurston's boundary for the Teichm\"uller space ${\mathcal T}(S)$ of $S$ defined exactly as before (but now absent any boundary constraints on isotopies).  Namely, a measured foliation of $S$ is said to have ``compact support'' if it has no 
peripheral leaves and no leaf asymptotic to a puncture.  Let ${\mathcal{MF}}_0(S)\subseteq{\mathcal {MF}}(S)$ denote the subspace of 
measured foliations of compact support with corresponding projective space ${\mathcal {PF}}_0(S)$.

\begin{theorem} [Generalized Thurston Boundary for possibly punctured surfaces]\label{thm:TBp}
 Fix a pants decomposition ${\mathcal P}$ of a possibly non-orientable and possibly punctured surface $S$.
Then ${\mathcal{PF}}_0(S)$ is thus a sphere
 of dimension $\#{\mathcal P}_1+2\#{\mathcal P}_2-1$, which has a natural piecewise-linear structure independent of the pants decomposition
 and provides a boundary $\overline{\mathcal{T}}(S)={\mathcal T}(S)\cup {\mathcal{PF}}_0(S)$ to the open ball that is Teichm\"uller space ${\mathcal T}(S)$ so as to form a closed ball
 $\overline{\mathcal{T}}(S)$.  The usual
 action of the mapping class group of $S$ on ${\mathcal T}(S)$ extends continuously to $\overline{\mathcal T}(S)$ by its natural action on 
 ${\mathcal{PF}}_0(S)$.
 \end{theorem}
 
 Let us also emphasize that there is a hybrid theory that includes both
 punctures and boundary components, each of the latter oriented and 
 containing a distinguished point.


\begin{thebibliography}{ABCD}



\bibitem{Abikoff} W. Abikoff, {\it The real analytic theory of Teichm\"uller space}, Lecture Notes in Math., Vol. 820, Springer, Berlin, 1980.

\bibitem{Klein1}
N.\ L.\ Alling and N.\ Greenleaf, Foundations of the theory of Klein surfaces, {\it Lecture Notes in 
Math.} {\bf 219},  Springer (1971). 

\bibitem{AN} A. Alexeevski and S. Natanzon, Noncommutative two-dimensional topological field theories and Hurwitz numbers for real algebraic curves, {\it Selecta Math.} {\bf 12} (2007),  307--377.; math.GT/0202164 (2002).

%\bibitem{AK} F. Atalan and M. Korkmaz, Automorphisms of curve complexes on non-orientable surfaces, preprint, 2012.


\bibitem{Chillingworth1969} D. R. J. Chillingworth,  
A finite set of generators for the homeotopy group of a non-orientable surface, {\it Proc.\ Camb\. Phil.\ Soc.} {\bf  65} (1969), 409--430.

\bibitem{DN}  C. Danthony and A. Nogueira,  
Measured foliations on non-orientable surfaces.  {\it Ann.\ Sci. \'Ec.\ Norm.\ Sup\'er.} {\bf 23} (1990), 469--494.

\bibitem {FLP} A. Fathi, F. Laudenbach, V. Po\'enaru, {\it Travaux de
Thurston sur les Surfaces},  { Asterisque} {\bf 66-67}, Soc.
Math. de France, Paris (1979).

%\bibitem{H1988} A. E. Hatcher, Measured foliation spaces for surfaces from the topological viewpoint. {\it Top.\ Appl.} {\bf 30} (1988), 63--88.

\bibitem{HT}
A. Hatcher and W. Thurston, A presentation for the mapping class
group of a closed oriented surface, {\it Topology} {\bf 19} (1980),
221--237.

\bibitem{Humphries}
S. P.  Humphries, Generators for the mapping class group, in: Topology of low-dimensional manifolds (Proc. Second Sussex Conf., Chelwood Gate, 1977), pp. 44--47, Lecture Notes in Math. {\bf 722}, Springer, Berlin, 1979. 

%\bibitem{KP} R.\ Kaufmann and R.\ C.\ Penner, Closed/open string diagrammatics, {\it Nucl.\ Phys.\ B} {\bf  748} (2006), 335--379.

\bibitem{Korkmaz1998} M. Korkmaz,   First homology group of mapping class groups of non-orientable surfaces, {\it Math.\ Proc.\ Camb\. Phil.\ Soc.} {\bf 123} (1998),  487--499.

\bibitem{Korkmaz2002} M. Korkmaz,  
Mapping class groups of non-orientable surfaces,
{\it Geom. Dedicata} {\bf  89} (2002), 107--133.

%\bibitem{Lauda}A.\ Lauda and H.\ Pfeiffer, Open-closed strings: Two-dimensional extended TQFTs and Frobenius algebras, {\it Top.\ Appl.} {\bf 155} (2008), 623--666.


\bibitem{Lickorish} W.\ B.\ R.\ Lickorish, A finite set of generators for the homeotopy group of a 2-manifold,
{\it Math.\ Proc.\ Camb.\ Phil.\ Soc.} (1964), 769--778.

\bibitem{Lickorish65} 
---, 
On the homeomorphisms of a non-orientable surface, {\it Math.\ Proc.\ Camb.\ Phil.\ Soc.} {\bf  61} (1965), 61--64. 

 \bibitem{Mumford1967} D. Mumford, 
Abelian quotients of the Teichm\"uller modular group, {\it J.\ Anal.\ Math.} {\bf  18} (1967), 227--244.


\bibitem{Klein2}
S.\ M.\ Natanzon, Klein surfaces, {\it Russian Math. Surveys} {\bf 45} (1990), 53--108. 


\bibitem{PPY} A. Papadopoulos, R.\ C. Penner  and O. Yurttas, On the action of the mapping class group of a non-orientable surface on arcs and curves,
in preparation.




  
\bibitem{Penner-thesis}
R. C. Penner, The action of the mapping class group on isotopy classes of curves and arcs in 
surfaces,  thesis,  Massachusetts Institute of Technology (1982), 180 pages.

\bibitem{Penner-bordered}
---,
Decorated Teichm\"uller space of bordered
surfaces,  {\it Comm.\ Anal.\ Geom.} {\bf 12} (2004), 793--820.


\bibitem{Harer-Penner} R. C. Penner  with J.\ L.\ Harer,  {\it Combinatorics of Train Tracks},
Annals of Mathematical Studies  {\bf 125},  Princeton Univ.\ Press (1992); second printing (2001).




\bibitem{Powell1978} J. Powell,  Two theorems on the mapping class group of a surface, {\it Proc.\ Am.\ Math.\ Soc.} {\bf  68} (1978), 347-350.




\bibitem{Scharlemann} M. Scharlemann,  
The complex of curves on non-orientable surfaces,
 {\it J.\ Lond.\ Math.\ Soc.}, II. (1982), 171--184.
 
 \bibitem{Seppala}
  M.\ Sepp\"al\"a, Teichm\"uller spaces of Klein surfaces, {\it Ann.\ Acad.\ Sci.\ Fenn.\ Ser.\ A I.} Mathematica Diss.\ {\bf 15} (1978), 1 - 37. 
 
\bibitem{Szep}
B. Szepietowski, The Mapping Class Group of a non-orientable Surface is Generated by Three Elements and by Four Involutions,
{\it Geom.\ Dedicata}
{\bf 117} (2006), 1--9.

\bibitem{Thurston}  W. P. Thurston, \emph{Three-Dimensional Geometry and Topology}, Volume 1, Princeton University Press, Princeton, New Jersey, 1997.
 
\bibitem {FN}
S. A. Wolpert, The Fenchel-Nielsen deformation, {\it Annals of Math} {\bf 115} (1982), 501--528.

\bibitem{FN2}
---, {\it Families of Riemann Surfaces and Weil-Petersson Geometry}, CBMS Series, Amer.\ Math. Soc. (2010).


\end{thebibliography}
\end{document}